\documentclass[11pt, reqno]{amsart}% use "amsart" instead of "article" for AMSLaTeX format
\usepackage[truedimen,margin=29truemm]{geometry}% See geometry.pdf to learn the layout options. There are lots.
\usepackage{graphicx}% Use pdf, png, jpg, or eps§ with pdflatex; use eps in DVI mode
\usepackage{bm}
\usepackage{amssymb}
\usepackage{amsmath}
\usepackage{amsthm}
\usepackage{mathtools}
\usepackage{url}
\usepackage{bigints}
\usepackage{enumerate} 
\usepackage{xcolor}

\newcommand{\DD}{\mathop{\mathrm{D}}}
\newcommand{\VOL}{\mathop{\mathrm{Vol}}}

\newcommand{\sgn}{\mathop{\mathrm{sgn}}} 
\newcommand{\ac}{\mathrm{ac}}

\newcommand{\re}{\mathop{\mathrm{Re}}} 
\newcommand{\im}{\mathop{\mathrm{Im}}} 

\newcommand{\supp}{\mathop{\mathrm{supp}}}

\newcommand{\N}{\mathbb{N}} 

\newcommand{\R}{\mathbb{R}}

\newcommand{\jap}[1]{\langle #1 \rangle}
\def\pa{\partial}
\def\a{\alpha}
\def\b{\beta}
\def\c{\gamma}
\def\d{\delta}
\def\e{\varepsilon}
\def\f{\varphi}
\def\g{\psi}

\def\l{\lambda}
\def\m{\mu}

\def\s{\sigma}

\def\x{\xi}

\numberwithin{equation}{section}

\theoremstyle{plain}%default
\newtheorem{thm}{Theorem}[section]
\newtheorem{proposition}[thm]{Proposition}
\newtheorem{lemma}[thm]{Lemma}

\theoremstyle{definition}

\newtheorem{example}[thm]{Example}

 \newtheorem{remark}[thm]{Remark}

 \newtheorem*{remarks*}{Remarks}
\newtheorem*{remark*}{Remark}

\title{Modified wave operators are unbounded on $L^p$ for $p \neq 2$}
\author{Akitoshi Hoshiya}
\address{Graduate School of Mathematical Sciences, The University of Tokyo, 3-8-1 Komaba, Meguro-ku, Tokyo 153-8914, Japan}
\email{hoshiya@ms.u-tokyo.ac.jp}

\author{Kouichi Taira}
\address{Faculty of Mathematics, Kyushu University, 744, Motooka, Nishi-ku, Fukuoka 819-0395, Japan}
\email{taira.kouichi.800@m.kyushu-u.ac.jp}
\begin{document}



\maketitle
%\section{}
%\subsection{}

\begin{abstract}

We study the $L^p$-boundedness of Isozaki--Kitada type modified wave operators associated with Schr\"odinger operators $P=-\pa_x^2 + V$ in one dimension for potentials $V$ including long-range potentials such as negative Coulomb-like ones $V(x) = -(1+|x|^2)^{-\mu/2}$ for $\mu>0$. For such potentials, if $\mu \in (1, 2)$, we prove that the middle and the high energy parts are bounded on $L^p (\R)$ for any $p \in [1, \infty]$ but the low energy part is unbounded except for $p=2$. Moreover, if $\mu \in (0, 1]$, we show that even the middle energy part is unbounded except for $p=2$. Actually, regarding the $L^p$-boundedness in the middle and the high energy regimes, we give a complete classification of slowly decaying potentials without imposing any sign condition. 
\end{abstract}

\section{Introduction and Main results}
\subsection{Introduction}\label{26662043}
This is a continuation of our analysis \cite{HT} on Schr\"odinger operators with long-range/slowly decaying potentials, in particular with negative Coulomb-like potentials. Let $\Delta$ be the standard (nonpositive) Laplacian on $L^2 (\R^d)$ with $d \ge 1$ and we consider Schr\"odinger operators $P:=-\Delta + V$, where a potential $V \in C^{\infty} (\R^d; \R)$ satisfies
\begin{align}
|\pa^{\alpha} _x V(x)| \lesssim \jap{x}^{-\mu -|\alpha|} \label{26641742}
\end{align}
for some $\mu \in (0, 2)$ and any $\alpha \in \N^d _0$. Here we use a standard notation $\jap{x}:= \sqrt{1+|x|^2}$. It is well-known that the restriction of $P$ on $C^{\infty} _c (\R^d)$ has a unique self-adjoint extension with its domain $H^2 (\R^d)$ and it is also denoted by the same symbol $P$. Among these potentials, of particular interest is a negative Coulomb-like one $V(x) = -\jap{x}^{-1}$ in three dimension, in which case $P$ is an analogue of the Hamiltonian of the Hydrogen atom $H=-\Delta -|x|^{-1}$. In spite of its physical importance, mathematical analysis of the time-dependent Schr\"odinger equation associated with $H$ (even with $P$) is technically difficult and quantitative estimates for the solution are still largely unknown. Many of the difficulties mainly come from the completely different nature of the spectra of $H$ and $P$ at low energy from $P_0 :=-\Delta$. In fact, the spectra of $H$ and $P$ consist of the absolutely continuous part $[0, \infty)$ and infinitely many negative discrete eigenvalues accumulating to $0$, though the spectrum of $P_0$ is purely absolutely continuous \cite{RS}. Moreover, the asymptotics of generalized eigenfunctions are also different at low energy \cite{Su, HT}.

Recently, in one dimension, the authors proved the dispersive estimates \cite[Theorem 1.3]{HT} for the Schr\"odinger equation, namely $|e^{-itP} E_{\ac} (P) (x, x')| \lesssim |t|^{-\frac{1}{2}}$ if $V$ satisfies (\ref{26641742}) and $-V(x) \gtrsim \jap{x}^{-\mu}$. Here the structure of the spectrum of $P$ is just the same as the above $H$ and $E_{\ac} (P)$ denotes the orthogonal projection onto the absolutely continuous subspace. This is the first result on quantitative pointwise bounds for the Schr\"odinger propagator with negative slowly decaying potentials, though it is a decisive tool for well-posedness and scattering theory of nonlinear Schr\"odinger equations. For later convenience, we briefly recall outcomes and the strategy of the proof. The proof consists of two parts: The first step is a detailed asymptotic analysis of generalized eigenfunctions of $P$, including estimates on their derivatives with respect to an energy parameter. Then, with functional calculus and scattering theory, we can use them to write the Schr\"odinger propagator as a superposition of oscillatory integral operators with explicit phase functions. In the second step, using techniques from harmonic analysis like degenerate stationary phase theorems, we give a pointwise bound on such oscillatory integrals. See \cite{HT} for more details. At the end of this paragraph, we mention a recent work \cite{GJZ}, where dispersive estimates for $H$ are obtained in three dimension.

In this paper, as a next step of our study, we consider the $L^p$-boundedness of modified wave operators in one dimension associated with $P$. If the potential is of short range, i.e., $\mu >1$ in (\ref{26641742}), it is well-known that wave operators
\begin{align}
\tilde{W}_{\pm} := \operatorname*{s-lim}_{t \to \pm \infty} e^{itP} e^{-itP_0} \label{2665025}
\end{align}
exist as bounded operators on $L^2 (\R^d)$ and they are asymptotically complete, which means the range of $\tilde{W}_{\pm}$ coincide with $\mathrm{Ran}~ E_{\ac} (P)$. On the other hand, if $V$ is of long range, i.e., $\mu \le 1$ in (\ref{26641742}), the standard wave operators (\ref{2665025}) do not exist in general and we have to introduce \textit{modified wave operators}. There are several kinds of modifications including H\"ormander type, Yafaev type and Dollard type modifiers (see \cite[Chapter XXX]{Hor} and \cite{DG, Y} for textbook presentations), which are related to each other (see Subsection \ref{eq:sevwavops1} and \cite[\S 6.4 and \S 6.5]{DS}).
Here we employ the Isozaki--Kitada modifier $J=J_{\pm}$, which is a time-independent one introduced in \cite{IK}, and it enables us to define the following modified wave operators:
\begin{align}\label{eq:IKmodwave}
W_{\pm}:= \operatorname*{s-lim}_{t \to \pm \infty} e^{itP} J_{\pm} e^{-itP_0}.
\end{align}
Strictly speaking, for each fixed $\lambda_0 >0$, the operators $J_{\pm}=J_{\pm, \lambda_0}$ are given by oscillatory integral operators
\begin{align*}
J_{\pm}f(x) := \frac{1}{2\pi}\int_{\R^{2}} e^{i(\phi_{\pm} (x, \xi) -y\xi)} a_{\pm}(x, \xi) f(y) dyd\xi.
\end{align*}
The precise construction is provided in Subsection \ref{subsection:IK} for the middle and high energy regimes and Subsection \ref{subsection:lowIsoKita} for the low energy regime under an additional assumption on $V$.
%for $f \in C^{\infty} _c (\R^d)$, where the phase functions $\phi_{\pm} = \phi_{\pm, \lambda_0} \in C^{\infty} (\R^{2d}; \R)$ are approximate solutions of the Eikonal equation $|\pa_x \phi_{\pm} (x, \xi)|^2 +V(x) = |\xi|^2$, $a_{\pm}=a_{\pm, \lambda_0}$ are symbols of order $(0, 0)$ in scattering calculus, and \eqref{eq:IKmodwave} satisfies $W_{\pm} E_{P_0} ([\lambda_0, \infty)) (L^2 (\R^d))= E_{P} ([\lambda_0, \infty)) L^2 (\R^d)$ and $\|W_{\pm} E_{P_0} ([\lambda_0, \infty)) f\| = \|E_{P_0} ([\lambda_0, \infty)) f\|$ for $f \in L^2 (\R^d)$. Actually we can take strong limits of $W_{\pm}$ as $\l_{0} \to +0$ and the obtained operators are isometries \cite{DS}.
Geometrically, $J_{\pm}$ are Fourier integral operators in the sense that the phase functions parametrize the Lagrangian submanifolds $\{(x, \pa_x \phi_{\pm}, y, \xi) \mid y= \pa_{\xi} \phi_{\pm} (x, \xi)\}$. See \cite{Y, T1} for a simple construction of the Isozaki--Kitada modifier. 
   
Our aim here is to determine for which $p \in [1, \infty]$ the modified wave operators in \eqref{eq:IKmodwave} are bounded on $L^p (\R)$. In these thirty years, the $L^p$-boundedness of wave operators has been intensively studied, mostly focusing on Schr\"odinger operators with rapidly decaying potentials. See \cite{AY, B, BS1, BS2, DF, EGG1, EGG2, FY, GG, HY, MSZ, Ni, W, Ya1, Ya2, Ya3, Ya4} and references therein. One of the motivations is that, once $\tilde{W}_{\pm} \in \mathcal{L} (L^q)$ and $\tilde{W}^* _{\pm} \in \mathcal{L} (L^p)$ are obtained, the intertwining property
\begin{align*}
f(P) E_{\ac} (P)= \tilde{W}_{\pm} f(P_0) \tilde{W}^* _{\pm}
\end{align*}
reduces the $L^p \to L^q$-boundedness of $f(P)E_{\ac} (P)$ to that of $f(P_0)$, where the latter is a Fourier multiplier and easier to handle. Due to such a reason, the $L^p$-boundedness of wave operators has been a strong tool in nonlinear theory \cite{NS}. 

Thanks to the aforementioned works, it is known that if $V$ has sufficiently rapid decay and high regularity, $\tilde{W}_{\pm}$ are bounded on $L^p$ for any $p \in (1, \infty)$ as long as the threshold of the spectrum is regular \cite{AY, B, BS1, BS2, DF, W, Ya1, Ya2, Ya3}. However, the endpoint cases $p =1,  \infty$ are quite subtle and $\tilde{W}_{\pm}$ are not necessarily bounded even if the threshold is regular \cite{W, MWY1}. Beyond the regular cases, obstructions at threshold energy can shrink the range of $p$ in many cases \cite{CSWY, FY, GG, MWY2, Ya4} but in some situations $\tilde{W}_{\pm}$ are bounded for any $p \in (1, \infty)$ under further stronger decay on $V$ or under certain orthogonality of $V$ and zero-energy eigenfunctions \cite{AY, DF, EGG2, GG, W}. It should be also notable that, without sufficient regularity on $V$, the range of $p$ can shrink even if $\supp V$ is compact \cite{EGG1}. Finally we mention the recent work \cite{MSZ}, where the Schr\"odinger operator $-\Delta + a|x|^{-2}$ with $a \ge -(d-2)^2/4$ is considered. In $d \ge 2$, it was shown that $\tilde{W}_{\pm}$ are bounded for $p \in (1, \infty)$ if $a \ge 0$ but the range of $p$ shrinks if $a <0$ due to the strong negative singularity at zero. See also \cite{DGJ,FSWZZ} for Schr\"odinger operators with critical electromagnetic potentials.
The corresponding phenomena for the dispersive estimates were recently investigated \cite{JZ, T2}.

Though the above results consider the case $\mu \ge 2$ in (\ref{26641742}), our focus is mainly on $\mu \in (0, 2)$ and on modified wave operators $W_{\pm}$. As far as the authors know, this is the first result for long-range potentials in the middle energy regime and for slowly decreasing potentials in the low energy regime. Our findings are completely different phenomena from the aforementioned cases, regarding Schr\"odinger operators with slowly decaying potentials, especially with $-\jap{x}^{-\mu}$, which are of interest physically.
\subsection{Main results}
%We split the exposition into two cases, one dimensional and higher dimensional cases, where a more detailed analysis can be possible in the former case.
%\subsubsection{One dimensional cases}
%In one dimension, we have explicit solutions of the Eikonal equation $y_{\pm} (x, \xi) = \pm\int_{0}^{x} \sqrt{\xi^2 -V(s)} ds$ as long as $\xi^2 -V(s) >0$ for $s \in [0, x]$. In particular $y_{\pm}$ are global solutions if $V <0$. First we deal with general potentials, not necessarily attractive ones.
First, we discuss the boundedness of modified wave operators in the middle and high energy regimes. The low energy part should be sensitive to the sign and the shape of $V$ and it seems impossible to give a unified analysis.
 Just for the sake of brevity, for a potential $V$ satisfying (\ref{26641742}), we fix $\lambda_0 >0$ such that 
\begin{align}\label{eq:lam0def}
\lambda^2 _0 > |V(x)|\quad \text{for all $x \in \R$},
\end{align}
and consider energies $\l\geq \sqrt{2}\l_0$. The following theorem remains valid for any fixed constant $\l_0>0$, appealing to the results in \cite[\S 3, \S5]{HT}. However, since the restriction \eqref{eq:lam0def} makes both the exposition and notation clearer, we focus on the case where \eqref{eq:lam0def} holds. 

%We discuss the boundedness of modified wave operators in the middle and high energy regimes $|\xi| > \l_0$. The low energy part should be sensitive to the sign and the shape of $V$ and it seems impossible to give a unified analysis. See also \S \ref{26662100}. \red{Using $y_{\pm}$, we define $\phi = \phi_{\pm}$ (independent of the sign $\pm$) by setting $\phi (x, \xi)=y_+ (x, \xi)$ if $\xi > \l_0$ and $\phi (x, \xi) = y_{-} (x, \xi)$ if $\xi <-\l_0$, which appears in the modifier $J$. In addition, we set the amplitude $a \equiv 1$ in $J$. TODO: $\phi_{\pm}$ should depend on $\pm$.} Then we have the $L^2$-convergence of $e^{itP} J e^{-itP_0} F(P_0)f$ as $t \to \pm \infty$ for any $F \in C^{\infty} \cap L^{\infty} (\R; \R)$ with $\supp F \subset (\l^2 _0, \infty)$ and for any $f \in \mathcal{S} (\R)$. Hence $W_{\pm} F(P_0)$ are well-defined bounded operators\footnote{Actually $W_{\pm} F(P_0)$ are equal to standard wave operators $\tilde{W}_{\pm} F(P_0)$ as long as $\mu >1$.} on $L^2 (\R)$. Our first result is

\begin{thm}[Middle and high energy wave operators]\label{26661518}
Let $V \in C^{\infty} (\R; \R)$ satisfy \eqref{26641742} for some $\mu>0$, $\l_0>0$ satisfy \eqref{eq:lam0def} and $W_{\pm}$ be the modified wave operators introduced in Subsection \ref{subsection:IK}. We set $M_{\DD} (x):= \frac{1}{2} \int_{0}^{x} V(s)ds$ for $x \in \R$. Then

\noindent$(i)$ \underline{\textbf{The case $M_{\DD} \in L^{\infty} (\R)$}} Let $\psi \in C^{\infty} (\R; [0, 1])$ be such that $\psi (\l) =1$ for $\l\gg 1$ and $\supp \psi \subset (2\l^2 _0, \infty)$. Then $W_{\pm} \psi (P_0)$ are bounded on $L^p (\R)$ for any $p \in [1, \infty]$.

\noindent$(ii)$ \underline{\textbf{The case $M_{\DD} \notin L^{\infty} (\R)$}} Let $p \in [1, \infty]$ and $\psi \in C^{\infty} _c (\R; \R)$ be such that $\supp \psi \subset (2\l^2 _0, \infty)$ and not identically zero. Then $W_{\pm} \psi (P_0)$ are bounded on $L^p (\R)$ if and only if $p=2$.
\end{thm}

\begin{remark}
\noindent$(1)$
D'Ancona-Fanelli \cite[Lemma 2.1]{DF} showed that $\tilde{W}_{\pm}\g(P_0)$ are bounded on $L^p$ for all $1\leq p\leq \infty$ under the assumption $V\in L^1(\R)$. Then, our result in $(i)$ for $\mu>1$ follows from their result and the equivalence of the $L^p$-boundedness of $W_{\pm}\g(P_0)$ and $\tilde{W}_{\pm}\g(P_0)$, which is stated in Lemma \ref{lem:waveL^pequiv} $(i)$. On the other hand, our results include weakly oscillating potentials which are not in $L^1$, as is discussed in Example \ref{26662124}.

\noindent$(2)$ By Theorem \ref{26661518} $(ii)$ and Lemma \ref{lem:waveL^pequiv} $(ii)$, Dollard's modified wave operators, Yafaev's modified operators, and H\"ormander's modified wave operators are not bounded on $L^p$ for $p\neq 2$ when $M_{\DD}\notin L^{\infty}$.

\end{remark}

Note that there is NO assumption on the decay rate $\mu >0$. Theorem \ref{26661518} implies even if $V$ is repulsive, i.e., $V(x) \gtrsim \jap{x}^{-\mu}$ for some $\mu \in (0, 1]$, modified wave operators are unbounded in the middle energy regime, though the dispersive estimates (without any loss) are valid \cite[\S 5]{HT}. 

Next we focus on attractive Coulomb-like potentials for which we give more detailed analysis including the low energy regime. We can define the well-behaved Isozaki--Kitada modified wave operators $W_{\pm}$ in the low-energy regime, as introduced in Section \ref{subsection:lowIsoKita} and motivated by the paper \cite{DS}. Strictly speaking, the modified wave operators $W_{\pm}$ considered here are not identical to those in Theorem \ref{26661518}. Nevertheless, they coincide on $\mathrm{Ran} E_{P_0}([2\l_0^2,\infty))$ and therefore we continue to use the same notation.

For a fixed constant $c>0$, let $\g_c \in C^{\infty} _c ([0, \infty); [0, 1])$ satisfy $\g_c (\lambda)=1$ for $\lambda \in [0, c]$ and $\supp \g_c \subset [0, 2c]$. %Let $\psi_{c} \in C^{\infty} (\R; [0, 1])$ satisfy $\psi_c (\lambda) = 1$ for $\lambda \in [2c, \infty)$ and $\supp \psi_c \subset [c, \infty)$. Then

\begin{thm}[Low energy wave operators for negative Coulomb-like potentials]\label{2665014}
Let $V \in C^{\infty} (\R; \R)$ satisfy \eqref{26641742} and $-V(x) \gtrsim \jap{x}^{-\mu}$ for some $\mu \in (0, 2)$ and $W_{\pm}$ be the modified wave operators introduced in Subsection \ref{subsection:lowIsoKita}. We fix any constant $c>0$. Then, for $p \in [1, \infty]$, $W_{\pm} \g_c (P_0)$ are bounded on $L^p (\R)$ if and only if $p=2$.

%\noindent$(i)$ \underline{\textbf{The case $\mu \in (1, 2)$}} $W_{\pm} \psi_c (P_0)$ are bounded on $L^p (\R)$ for any $p \in [1, \infty]$. On the other hand, for $p \in [1, \infty]$, $W_{\pm} \g_c (P_0)$ are bounded on $L^p (\R)$ if and only if $p=2$.
%
%\noindent$(ii)$ \underline{\textbf{The case $\mu \in (0, 1]$}} Let $p \in [1, \infty]$. For any $\psi \in C^{\infty} _c ((0, \infty); [0, 1])$ which is not identically zero, $W_{\pm} \psi (P_0)$ are bounded on $L^p (\R)$ if and only if $p=2$.
\end{thm}

\begin{remark}\label{26662136}
%\red{delete Before we turn to higher dimensional cases}, we give small remarks.

\noindent$(1)$ When $1<\mu<2$, $W_{\pm}$ here are bounded in the middle and high energy regimes by Theorem \ref{26661518}. Theorem \ref{2665014} demonstrates that the failure of the $L^p$-boundedness in this case is entirely due to their low energy behavior. 

\noindent$(2)$
There are some results dealing with $L^1 \to L^{1, \infty}$ or $L^{\infty} \to \mathrm{BMO}$ boundedness of wave operators, see e.g., \cite{HY, MWY1}. In our cases where the $L^p$-boundedness is true only for $p=2$, both $L^1 \to L^{1, \infty}$ and $L^{\infty} \to \mathrm{BMO}$ boundedness also fail by a simple combination of the $L^2$-boundedness, interpolation argument and proof by contradiction.

\noindent$(3)$ It seems that the assumption (\ref{26641742}) is needed only for $\alpha =0, 1, 2, 3, 4$. %\red{Jost解の表示を得るために2回必要で，2回部分積分している部分があるので$V$についての微分は多分$4$回必要．前の論文ではシンボルの1回微分しか使っていないので，$V$についての微分は多分$3$回必要だった？}
In fact, once the integral formula of modified wave operators in terms of wave matrices is justified, all the calculations in this paper work well for $C^4$ potentials in one dimension. 
\end{remark}

Note that zero is neither an eigenvalue nor a resonance in this case \cite{HT}. Nevertheless the $L^p$-boundedness is violated. This is in stark contrast to the case $\mu >2$ by D'Ancona--Fanelli \cite{DF}, in which standard wave operators are bounded on $L^p (\R)$ for $p \in (1, \infty)$ if zero is a regular point. Our result suggests slow decay and negativity of a potential lead to the failure of the boundedness even if there is no obstruction at threshold energy. However we should recall from \S \ref{26662043} that the dispersive estimates HOLD without any loss in this setting, hence Theorem \ref{2665014} reveals a gap between the $L^p$-boundedness of wave operators and the dispersive estimates.

\subsection{Examples}

In the following we discuss several types of potentials. 

\begin{example}
Let $V(x)=c\jap{x}^{-\mu}$ with $c\neq 0$. Then it satisfies \eqref{26641742}. Moreover, $M_{\DD}\in L^{\infty}(\R)$ if and only if $\mu>1$. Therefore, for $p\neq 2$, the modified wave operators $W_{\pm}$ are $L^p$-bounded in the middle and high energy regimes if and only if $\mu>1$. When $1<\mu<2$ and $c<0$, $W_{\pm}$ are unbounded on $L^p$ in the low enegy regime.

\end{example}

The next one satisfies (\ref{26641742}) for $\mu =1$ (and not for $\mu >1$) but we have the $L^p$-boundedness of modified wave operators. 

\begin{example}\label{26662124}
Let $V \in C^{\infty} (\R; \R)$ be such that $V(x)=(1+|x|)^{-1}\sin(\log(1+|x|))$ for $|x| \ge 1$. Then $V$ satisfies (\ref{26641742}) and $M_{\DD} \in L^{\infty} (\R)$. In fact,
\begin{align*}
\int_1^xV(s)ds=\int_1^x\frac{\sin(\log(1+s))}{(1+s)} ds=\int_{\log 2}^{\log (1+x)}\sin(t)dt=\cos(\log 2)-\cos(\log (1+x)) 
\end{align*}
for $x\geq 1$. The negative region can be estimated by the symmetry $V(x) =V(-x)$ for $|x| \ge 1$. Hence Theorem \ref{26661518} yields the $L^p$-boundedness of modified wave operators in the middle and high energy regimes for any $p \in [1, \infty]$.
\end{example}

The point in Example \ref{26662124} is that due to the logarithmic oscillation of $V$, it can be regarded as a short range potential when averaged. The next two examples are also oscillating potentials but modified wave operators are not bounded because of their weakness of oscillation. 

\begin{example}\label{26662236}
We define $V(x):=(\jap{x}^{1-\m}\sin(\log\jap{x}) )'$ for $\mu \in (0, 2)$. Then $V$ satisfies (\ref{26641742}) and $M_{\mathrm{D}}(x)=\jap{x}^{1-\m}\sin(\log\jap{x})$. Hence if $\mu \in [1, 2)$, we obtain $M_{\DD} \in L^{\infty} (\R)$ and the $L^p$-boundedness of modified wave operators in the middle and high energy regimes for any $p \in [1, \infty]$. On the other hand, if $\mu \in (0, 1)$, we have $M_{\DD} \notin L^{\infty} (\R)$. To see this, we can take a sequence $x_n \to \infty$ such that $M_{\mathrm{D}}(x_n)\to \infty$ as $ n\to \infty$ (just take $\log\jap{x_n}= \frac{\pi}{2}+2\pi n$). Therefore modified wave operators are unbounded on $L^p (\R)$ with $p \neq 2$ in the middle energy regime.
\end{example}

\begin{example}\label{26662259}
We fix $\mu \in (0, 1]$. Let $V \in C^{\infty} (\R; \R)$ be such that $V(x)=(1+C_{\mu}\sin(\log|x|))|x|^{-\mu}$ for $|x|\geq 1$, which satisfies (\ref{26641742}). Here we set $C_1 = 2$ and $C_{\mu} \in (1, (2-\mu)^{-1} (1-\mu + (1-\mu)^{-1}) )$ for $\mu \in (0, 1)$. Though $V$ changes the signature infinitely many times due to the oscillation, we have $M_{\DD} \notin L^{\infty} (\R)$. Indeed, for $\mu=1$,
\begin{align*}
\int_1^xV(s)ds=(\log x)-2\cos(\log x)+2\quad (x\geq 1)
\end{align*}
and hence $\int_1^xV(s)ds\gtrsim \log x$ for $x\gg 1$. For $0<\mu<1$, 
\begin{align*}
\int_1^xV(s)ds=\frac{1}{1-\mu}(x^{1-\mu}-1)+\frac{C_{\mu} x^{1-\mu}((1-\mu)\sin(\log x)-\cos(\log x))+C_{\mu}}{(1+(1-\mu)^2)},
\end{align*}
where we use $\int_1^x\frac{\sin(\log x)}{x^{\mu}}dx=\int_0^{\log x}(\sin u)e^{(1-\mu)u}du=\frac{1}{1+(1-\mu)^2}[e^{(1-\mu)u}((1-\mu)\sin u-\cos u)]_{u=0}^{u=\log x}$. Since $(1-\mu)^{-1}> C_{\mu} (2-\mu)(1+(1-\mu)^2)^{-1}$, we have  $\int_1^xV(s)ds\gtrsim x^{1-\mu}$ for $x\gg 1$. Therefore modified wave operators are unbounded on $L^p (\R)$ with $p \neq 2$ in the middle energy regime.
\end{example}

%\subsubsection{Higher dimensional cases}
%\begin{thm}\label{2665015}
%Let
%\end{thm}
%By the invariance principle, the $L^p$-unboundedness of $W_{\pm} \psi (P_0)$ implies that of $W_{\pm} (\sqrt{P}, \sqrt{P_0}) \psi (P_0)$, which are modified wave operators for the pair $(\sqrt{P}, \sqrt{P_0})$.

\subsection{Further observations}\label{26662100}
\underline{\textbf{Repulsive short range potentials}}
It should be interesting to consider the $L^p$-boundedness for $P=-\pa^2 _x + V(x)$ in the low energy regime, where $V$ satisfies (\ref{26641742}) and $V(x) \gtrsim \jap{x}^{-\mu}$ for some $\mu \in (1, 2)$. In fact, it has been observed that $P$ satisfies better estimates than $P_0$ at low energy. For example, uniform resolvent estimates hold with slowly decaying weights \cite{N1}, and the Riesz operator $P^{-1}$ is actually a pseudodifferential operator, not a singular integral operator \cite{SW, T4}. Moreover, in higher dimensions, the Strichartz and dispersive estimates are proved without any loss \cite{BTVZ, M, T3}. Hence (modified) wave operators might be bounded on $L^p (\R)$ for $p \neq 2$, which are contrast with negative Coulomb-like potentials. See also \cite[\S 1.4]{HT} for more technical comments.

\noindent
\underline{\textbf{Oscillating short range potentials}}
Though modified wave operators are bounded in the middle and high energy regimes for $V$ in Examples \ref{26662124} and \ref{26662236} with $\mu \in [1, 2)$, the low energy regime is completely unknown. This is because our construction of generalized eigenfunctions in \cite[\S 3]{HT}, in particular the reduction with Liouville's transform is not suitable for such models at low energy. Possibly we could employ a construction of the Jost functions in \cite[Vol. III]{RS} for oscillating potentials but more works are needed to obtain uniform bounds of (derivatives of) the Jost functions with respect to an energy parameter.

%\noindent
%\underline{\textbf{Boundedness in higher dimensions}} \red{後々のことを考えるとNo comment の方がいいかも?任せます}
%In higher dimensions, it is non-trivial to show the $L^p$-boundedness even if we assume (\ref{26641742}) for some $\mu \in (1, 2)$. In one dimension, we know the precise asymptotic structure of the wave matrix (or the scattering matrix) near spatial infinity in any energy regime. This is crucial when we prove Theorem \ref{26661518} $(i)$. In higher dimensions, wave matrices are integral operators with generalized eigenfunctions in their kernels, which are no more solutions to ODEs and harder to analyze. Perhaps asymptotic expansions of outgoing/incoming resolvents in \cite{Su} could be useful. We leave this problem as a future work.

\noindent
\underline{\textbf{Applications}}
Since modified wave operators also enjoy intertwining properties \cite{IK}, we can use Theorem \ref{26661518} $(i)$ to deduce the dispersive, Strichartz, uniform Sobolev estimates in the middle and high energy regimes for many dispersive/hyperbolic equations. Proofs of such applications are well-known \cite{MSZ} and we do not repeat them here.

\subsection{Idea of the proof}
Due to the slow decay of $V$, perturbative arguments like the Born expansion \cite{DF, Ya1} do not work well even in the middle and high energy regimes and even for positive results. Moreover, in the standard stationary representation formula for the usual wave operators \cite{CSWY, DF, EGG2, FY, GG, Ya1, Ya2, Ya3, Ya4}
\begin{align*}
\tilde{W}_{+} = \mathrm{Id} - \frac{2}{\pi} \int_{0}^{\infty} R(\l^2 -i0) V \im R_0 (\l^2 +i0) \l d\l,
\end{align*}
the expression for the second term in terms of the Jost functions is rather complicated.
%becomes quite complicated due to unusual oscillations in the generalized eigenfunctions
%\begin{align*}
%u_{\pm} (x, \lambda) =\frac{1}{(\l^2-V(x))^{\frac{1}{4}}} \left(\tilde{a}_{\pm, +} (x, \lambda) e^{iy(x,\l)} + \tilde{a}_{\pm, -} (x, \lambda) e^{-iy(x,\l)}\right)
%\end{align*}  
%of $-\partial^2 _{x} u_{\pm} (x, \lambda) +V(x)u_{\pm} (x, \lambda) = \lambda^2 u_{\pm} (x, \lambda)$ studied in \cite[\S 3]{HT}. %if $-V(x) \gtrsim \jap{x}^{-\mu}$ and (\ref{26641742}) are satisfied \cite[Theorem 3.11]{HT}. 
We further have to take into account the modifier $J$ since we consider not $\tilde{W}_{\pm}$ but modified wave operators $W_{\pm}$. To overcome this difficulty, instead, we employ the formula
\begin{align*}
W_{\pm} = \int_{0}^{\infty} W_{\pm} (\l) \mathcal{F}_0(\l) d\l,
\end{align*}
where $W_{\pm} (\l)$ are wave matrices\footnote{These operators are also called the Poisson operators in geometric scattering theory and the Fourier extension operators in harmonic analysis. Moreover, their inverses are called the distorted Fourier transforms or the generalized Fourier transforms in the traditional scattering theory.} of $P$ and $\mathcal{F}_{0} (\l)$ is the Fourier restriction operator. Since $\mathcal{F}_{0} (\l)$ are simple explicit operators and $W_{\pm} (\l)$ can be expressed in terms of $u_{\pm}$, we obtain a representation formula of the integral kernel of $W_{+}$ as
\begin{align*}
W_{+} (x, x') = \sum_{\sigma_1, \sigma_2 \in \{\pm\}} \int_{0}^{\infty} b_{\sigma_1,\s_2} (x, \l) e^{i\sigma_2 y(x, \l) -i\sigma_1 \l x'} d\l =: \sum_{\sigma_1, \sigma_2\in \{\pm\}} W_{+,\s_1,\s_2} (x, x'),
\end{align*}
where $b_{\s_1,\s_2}$ are products of $\tilde{a}_{\pm, \pm}$ and explicit functions and satisfy $|\pa_{\l}^{\a}b_{\s_1,\s_2}(x,\l)|\lesssim |\l|^{-|\a|}$ for $\a=0,1,2$. See the proof of Propositions \ref{prop:waveexpression1} and \ref{prop:waveexpression2}. 

To study the behavior of the modified wave operator $W_{+}$ at the middle and high energies for Theorem \ref{26661518}, we observe that in the regime $\l > \l_0$, the phase function satisfies
\begin{align*}
y(x, \lambda) = \lambda \int_{0}^{x} \left(1-\frac{V(s)}{\lambda^2}\right)^{\frac{1}{2}} ds = x\l - M_{\DD} (x)\l^{-1} + R(x, \lambda),
\end{align*}
where we recall $M_{\DD}(x)=\frac{1}{2}\int_0^xV(s)ds$ and the term $R$ satisfies $|\partial^{\alpha} _{\lambda} R (x, \lambda)| \lesssim \frac{1}{\lambda^{3+\alpha}} \int_{0}^{x} |V(s)|^{2} ds$
for any $\alpha \in \N_0$. Then,

\noindent$(1)$ We can regard $R$ as a remainder of the principal term $x\l-M_{\DD}(x)\l^{-1}$, which is easy at least when $V\in L^2(\R)$ or $M_{\DD}\in L^{\infty}(\R)$ (see Lemma \ref{26542109}), and is possible even when $\m>0$ and $M_{\DD}\notin L^{\infty}(\R)$ along an appropriate sequence $x_n$ with $|x_n|\to \infty$ and $|M_{\DD}(x_n)|\to \infty$, see Lemma \ref{2658032}.

\noindent$(2)$ Looking at the principal term $x\l-M_{\DD}(x)\l^{-1}$, we can also regard the second term $M_{\DD}(x)\l^{-1}$ as a remainder when $M_{\DD}\in L^{\infty}(\R)$. In this case, $W_+$ is a sum of oscillatory integral operators with the phase functions $\s_2\l x-\s_1\l x'$ for $\s_1,\s_2\in \{\pm\}$, whose $L^p$-boundedness can be studied by the theory of Calder\'on-Zygmund operators at least when $1<p<\infty$. To deal with the end point cases $p=1$ and $p=\infty$, we need an additional argument, see the proof of Theorem \ref{264112215} and Remark \ref{rem:CalZyg}.

\noindent$(3)$ When $M_{\DD}\notin L^{\infty}(\R)$, we have to take the term $M_{\DD}(x)\l^{-1}$ into account more carefully. Considering the case $(\s_1,\s_2)=(+,+)$, the phase function $\l x-\l^{-1}M_{\DD}(x)-\l x'$ has a non-degenerate critical point $\l_{x,x'}=\sqrt{-\frac{M_{\DD}(x)}{x-x'}}$ when $x-x'\sim -M_{\DD}(x)$ holds, and the stationary phase theorem shows that $W_{+,+.+}(x,x')$ has an expression
\begin{align*}
W_{+,+.+}(x,x')=|x-x'|^{-\frac{1}{2}}\tilde{b}(x,x')e^{i|x-x'|\Phi(x,x')} +O(|x-x'|^{-1}),
\end{align*}
where $\tilde{b}(x,x')=b_{+,+}(x,\l_{x,x'})$ and $\Phi(x,x')=|x-x'|^{-1}(y(x,\l_{x,x'})-\l_{x,x'}x')$. Now we recall the characterization of the operator norms in $L^1$ and $L^{\infty}$:
\begin{align*}
\|T\|_{L^1(\R)\to L^1(\R)}=\sup_{x'\in \R}\int_{\R}|T(x,x')|dx,\quad \|T\|_{L^{\infty}\to L^{\infty}}=\sup_{x\in \R}\int_{\R}|T(x,x')|dx'.
\end{align*}
From the above expression, they blow up for $T(x,x')=W_{+,+.+}(x,x')1_{(x-x')\sim -M_{\DD}(x)}$ and $W_{+,+,+}$ is unbounded on $L^1$ and $L^{\infty}$. To deal with the unboundedness for $1<p<\infty$, we observe that the suprema in the operator norms $\|T\|_{L^1\to L^1}$ and $\|T\|_{L^{\infty}\to L^{\infty}}$ are expected to be attained by a sequence of $L^1$-functions converging to a delta function and $u(x')=\sgn\overline{T(x_0,x')}$ for some $x_0\in \mathbb{R}$, respectively. Renormalizing these sequences in $L^p$ sense, we can construct an appropriate sequence $u_n \in L^p(\R)$ such that $\|W_{+, +, +} \psi (P_0) u_n\|_{L^p (\R)} / \|u_n\|_{L^p (\R)} \to \infty$. For a detail, see the proof of Theorems \ref{264131227} and \ref{264221222}. We also have a similar expression for $W_-$. We also note that each $W_{+,\s_1,\s_2}$ have quite different oscillatory behavior and the sequence $u_n$ can be chosen so that $\|W_{+, +, -} \psi (P_0) u_n\|_{L^p (\R)} / \|u_n\|_{L^p (\R)} $ and $\|W_{+, -, \s_2} \psi (P_0) u_n\|_{L^p (\R)} / \|u_n\|_{L^p (\R)}$ remain bounded in $n$.

%Now, there happens a problem that the second term $M_{\DD} (x)\l^{-1}$ cannot control the remainder term $R$ uniformly in $x \in \R$ due to the possibility of oscillation. Nevertheless, to overcome this,  we find a sequence $\{x_n\} \subset \R$ such that $|x_n| \to \infty$, $|M_{\DD} (x_n)| \to \infty$ and $|M_{\DD} (x_n)|^{-1} |\int_{0}^{x_n} V(s)^2 ds| \to 0$, and around such $x_n$, we construct a profile $\{u_n\}$ in the same spirit as above.
% what makes the situation difficult is the complicated behavior of the phase function. 

On the other hand, to study the behavior in the low energy regime under the additional assumption $-V(x)\gtrsim \jap{x}^{-\mu}$, we employ Taylor's expansion of the phase function $y(x,\l)$ at $\l=0$:
\begin{align*}
y(x,\l)=M_0(x)+\frac{\l^2}{2}M_1(x)+\mathcal{R}_{x}(\l),\quad M_j(x):=\int_0^x|V(s)|^{-\frac{2j-1}{2}}ds,
\end{align*}
where $\mathcal{R}_{x}(\l)$ can be regarded as a remainder term in the regime $|x|\to \infty$ and $\l\to 0$ while $\l\ll \jap{x}^{-\m/2}$.
To prove Theorem \ref{2665014} for $p \in [1, 2)$, we take a sequence of rapidly decreasing functions $\{u_n\}$ whose supports escape to infinity and Fourier supports concentrate near $0$. We can show that $\|W_{+, +, +} \psi (P_0) u_n\|_{L^p (\R)} / \|u_n\|_{L^p (\R)} \to \infty$ by considering the region where $M_1(x)\l^2\ll 1$ though other terms are bounded. 
On the other hand, if $p \in (2, \infty]$, we need to construct a completely different sequence which cancels the oscillation in $W_{+, +, +} (x, x')$ successfully, and leads to $\|W_{+, +, +} \psi (P_0) u_n\|_{L^p (\R)} / \|u_n\|_{L^p (\R)} \to \infty$. In this case, we have to choose the supports of $\{u_n\}$ carefully, otherwise weak oscillation could survive and make $\|W_{+, +, +} \psi (P_0) u_n\|_{L^p (\R)}$ small. %Regarding negative results in Theorem \ref{2665014} $(i)$, the strategy is similar but we have to choose a sequence $\{u_n\}$ whose supports escape to infinity and Fourier supports approaches to zero, where these two escaping/accumulating speeds are not independent and have to be chosen appropriately. Thanks to the assumption $\m<2$, such a localization does not violate the uncertainty principle.

%\red{TODO改行しない？}
%
%
%\red{In higher dimensions, we employ a similar integral expression
%\begin{align*}
%W_{\pm} = J_{\pm} + \int_{0}^{\infty} R(\l \mp i0) T_{\pm} (\l) W_{0} (\l)^* d\l.
%\end{align*}
%If we multiply pseudodifferential operators $Q_{\pm}$ from left, which are microlocally supported in outgoing/incoming regions, microlocal resolvent estimates \cite{IK2, IK3, N} imply $Q_{\pm} W_{\pm}$ are essentially $Q_{\pm} J_{\pm}$ modulo regularizing operators. Therefore our analysis is reduced to the $L^p$-boundedness of (microlocalized) Isozaki--Kitada modifiers as in one dimensional cases, and the aforementioned strategy is available, though more works are needed.}

\subsection{Notations}
We write $\N =\{1, 2, \dots\}$ and $\N_0 = \N \cup \{0\}$. The Dirac delta function and the Fourier transform are denoted by $\delta$ and $\mathcal{F}$ respectively. For $r \in \R$, we set $\R_{\ge r} :=\{x \in \R \mid x \ge r\}$ and $\R_{\le r}:= \{x \in \R \mid x \le r\}$. For $S \subset \R^d$, $1_{S}$ denotes the indicator function of $S$. Moreover, we write $A\lesssim B$ if there is a non-essential constant $C>0$ such that $A\leq CB$ and $A\sim B$ if both $A\lesssim B$ and $B\lesssim A$ hold. 

\subsection*{Acknowledgment}
AH is partially supported by FoPM, WINGS Program, the University of Tokyo and JSPS Research Fellowship for Young Scientists KAKENHI Grant Number JP25KJ0736. KT is partially supported by JSPS KAKENHI Grant Number JP23K13004. He is grateful to Tomoya Kato for fruitful discussions. The proof of Lemma \ref{26542109} was provided by ChatGPT-5 and verified by the authors.

\section{Modified wave operators and their oscillatory integral representations}

We assume that $V$ satisfies \eqref{26641742} in the rest of this paper and $\l_0>0$ is a constant such that \eqref{eq:lam0def} holds until the end of Section \ref{Section:mid}.

\subsection{Isozaki--Kitada modifiers and modified wave operators}\label{subsection:IK}

We recall modified wave operators constructed by Isozaki--Kitada \cite{IK}. In one-dimension, we can use explicit solutions of the eikonal equation as phase functions.

Let us recall $\l_0>0$ satisfy \eqref{eq:lam0def}. We define
\begin{align}\label{eq:yphidef}
y (x, \xi) :=\int_{0}^{x} \sqrt{\xi^2 -V(s)} ds,\quad \phi_{\pm}(x,\x):=\left\{\begin{aligned}
&\pm y (x, \xi) &&(\x>0)\\
&\mp y (x, \xi) &&(\x<0),
\end{aligned}\right.
\end{align}
which are well-defined and smooth as long as $|\x|\geq \l_0$. Then, the phase functions $\phi_{\pm}$ satisfy the eikonal equation $|\pa_x\phi_{\pm}(x,\x)|^2+V(x)=|\x|^2$, and $ \phi_{\pm}(x,\x)=x\x+o(|x|)$ whenever $|x|\to \infty$ and $\sgn(x\x)=\pm 1$. Let $\overline{\chi}\in C^{\infty}(\R;[0,1])$ be such that $\overline{\chi}(x)=1$ for $x\geq \sqrt{2}$ and $\overline{\chi}(x)=0$ for $x\leq 1$. We define $a_{\pm}(x,\x):=\overline{\chi}(\pm x/R_0)\overline{\chi}(\x/\l_0)+\overline{\chi}(\mp x/R_0)\overline{\chi}(-\x/\l_0)$ for a fixed $R_0\gg 1$ and the Isozaki--Kitada modifiers by 
\begin{align}\label{eq:Jpmdef}
J_{\pm}f(x) := \frac{1}{2\pi}\int_{\R^{2}} e^{i(\phi_{\pm} (x, \xi) -y\xi)} a_{\pm}(x, \xi) f(y) dyd\xi,
\end{align} 
which are bounded on $L^2(\R)$ by \cite[Lemma 3.3]{IK}. By the standard argument as in \cite{IK}, the limits $W_{\pm}$ in \eqref{eq:IKmodwave} exist and are unitary operators from $\mathrm{Ran}~ E_{P_0}([2\l_0^2,\infty))$ to $\mathrm{Ran}~ E_{P}([2\l_0^2,\infty))$, where $E_{P_0}$ and $E_P$ are the spectral projections. These ranges of $W_{\pm}$ are the absolutely continuous subspace of $P$ and $W_{\pm}$ have the intertwining property $PW_{\pm}=W_{\pm}P_0$. The operators $W_{\pm}$ are called \textbf{modified wave operators}.

\begin{remark}
The modified wave operators $W_{\pm}$ are independent of the choice of $R_0$ thanks to the local compactness of $P_0$.
\end{remark}

\subsection{Relationship to other wave operators}\label{eq:sevwavops1}
The results in this subsection are not needed for the proof of Theorems \ref{26661518} and \ref{26662136}.
Here, we discuss the relationship among the Isozaki--Kitada modified wave operators, the usual wave operators, Dollard's modified wave operators, Yafaev's modified wave operators, and H\"ormander's modified wave operators following \cite[\S 6.4 and \S6.5]{DS}, \cite[\S4.7 and \S 4.9]{DG} and \cite[\S 10.2]{Y}. The main difference from the setting in \cite{DS} is that they assume spherical symmetry of potentials (which corresponds to the evenness in the one-dimensional case). The resulting formulas \eqref{eq:modwaveusualwave} and \eqref{eq:modwaveDollard} relating these wave operators are a bit more complicated than those in \cite{DS}.

Let us recall that $\tilde{W}_{\pm}$ are the usual wave operators \eqref{2665025}, which can be defined when $\m>1$ and $W_{\DD,\pm}$ are Dollard's modified wave operators
\begin{align*}
W_{\DD,\pm}=\operatorname*{s-lim}_{t \to \pm \infty}e^{itP} e^{-itP_0}e^{-\frac{i}{2}\int_0^tV(sD_x)ds},
\end{align*}
which are defined when $\m>1/2$. Concerning Yafaev's modified wave operators, we use solutions $\Xi_{\pm}$ to the Hamilton--Jacobi equation (in the configuration space)
\begin{align}\label{eq:HamJac1}
\pa_t \Xi_{\pm} (t, x) + |\pa_x \Xi_{\pm} (t, x)|^2 +V(x) =0, \quad \epsilon \le \left|\frac{x}{2t}\right| \le \frac{1}{\epsilon}, \quad \pm t \ge T_{\epsilon},
\end{align}
which are constructed in the proof of Lemma \ref{lem:waveL^pequiv} $(iii)$ below for any $\epsilon >0$ and sufficiently large $T_{\epsilon} >0$, to define the modified free dynamics
\begin{align*}
(\mathcal{U}_{0, \pm} (t) f) (x)= e^{i\Xi_{\pm} (t, x)} (2it)^{-\frac{1}{2}} \hat{f} \left(\frac{x}{2t}\right)
\end{align*}
for any $\hat{f} \in C^{\infty} _c (\R \setminus \{0\})$. Now Yafaev's modified wave operators\footnote{In \cite[\S 1.5]{Y}, modified wave operators are introduced using approximate solutions to the Hamilton--Jacobi equation. As long as $\mu > 1/2$, they are equal to our Isozaki--Kitada modified wave operators $W_{\pm}$, which can be checked mimicking the proof of \cite[Theorem 10.2.6]{Y}. However, if $\mu \le 1/2$, they do not coincide in general because their appears non-vanishing error terms. It seems that if we employ Isozaki--Kitada modifiers with a phase function approximating the eikonal equation, which is constructed in \cite{T1}, they coincide for any $\mu >0$.} are defined for any $\mu >0$ by
\begin{align*}
W_{\mathrm{Y}, \pm} = \operatorname*{s-lim}_{t \to \pm \infty}e^{itP}\mathcal{U}_{0, \pm} (t) \bar{\chi} (|D_x /\l_0|),
\end{align*}
where $\overline{\chi}$ is introduced in Subsection \ref{subsection:IK}.
Finally, regarding H\"ormander's modified wave operators, we employ solutions $S_{\pm}$ to the Hamilton--Jacobi equation (in the momentum space)%\footnote{Strictly speaking, the Hamilton--Jacobi equation in \cite{DG} is $\pa_t S(t, \xi) = \frac{\xi^2}{2} + V(\pa_{\xi} S(t, \xi))$ since they consider Hamiltonians $-\frac{1}{2}\pa^2 _x + V$ instead of $-\pa^2 _x +V$.} 
\begin{align}\label{eq:HamJac2}
\pa_t S_{\pm} (t, \xi) = \xi^2 + V(\pa_{\xi} S_{\pm}(t, \xi)), \quad \epsilon \le |\xi| \le \frac{1}{\epsilon}, \quad \pm t \ge T_{\epsilon},
\end{align}
which are introduced in the proof of Lemma \ref{lem:waveL^pequiv} $(iv)$ below for any $\epsilon >0$ and sufficiently large $T_{\epsilon} >0$. Using these functions $S_{\pm}$, %\cite[Theorem 4.7.1 and Proposition 4.7.4]{DG} (see also \cite[Vol. IV]{Hor}), 
H\"ormander's modified wave operators
\begin{align*}
W_{\mathrm{H}, \pm} = \operatorname*{s-lim}_{t \to \pm \infty}e^{itP} e^{-iS_{\pm}(t, D_x)} \bar{\chi} (|D_x /\l_0|)
\end{align*}
are defined for any $\mu >0$.

\begin{lemma}\label{lem:waveL^pequiv}
Let $\chi\in C^{\infty}(\R)$ satisfy $\chi(\x)=1$ for $|\x|\gg 1$ and $\supp \chi\subset (\l_0,\infty)$. Then for any $p \in [1, \infty]$,
%Let $\eta \in C^{\infty} _c(\R)$ be such that $\supp \eta\subset (\l_0,\infty)$.  

\noindent$(i)$ The $L^p$-boundedness of $W_{\pm}\chi(D_x)$ and $\tilde{W}_{\pm}\chi(D_x)$ is equivalent when $\m>1$.

\noindent$(ii)$ The $L^p$-boundedness of $W_{\pm}\chi(D_x)$ and $W_{\DD,\pm}\chi(D_x)$ is equivalent when $\m>1/2$.

\noindent$(iii)$ The $L^p$-boundedness of $W_{\pm}\chi(D_x)$ and $W_{\mathrm{Y},\pm}\chi(D_x)$ is equivalent when $\m> 0$.

\noindent$(iv)$ The $L^p$-boundedness of $W_{\pm}\chi(D_x)$ and $W_{\mathrm{H},\pm}\chi(D_x)$ is equivalent when $\m> 0$.

\end{lemma}

\begin{proof}
Let $\g\in C^{\infty}(\R)$ satisfy $\g(\x)=1$ for $|\x|\gg 1$, $\supp \g\subset \{ \xi \in \R \mid |\xi| \in (\l_0,\infty)\}$, and $\g\chi=\chi$.
We write $\g_+= \g 1_{\R_{\geq 0}}$ and $\g_-=\g 1_{\R_{\leq 0}}$, which are smooth and satisfy $\g=\g_++\g_-$.

\noindent$(i)$ We only deal with the $+$ case and the other case is similarly handled.
We define $A_{\pm}(\x)=(\sgn \x)\int_0^{\pm\infty}(\sqrt{|\x|^2-V(s)}-|\x|)ds$ and $R_{\pm}(x,\x):=-(\sgn \x)\int_x^{\pm\infty}(\sqrt{\x^2-V(s)}-|\x|)ds$ for $\x^2\geq \l_0^2$ so that
\begin{align*}
\phi_+(x,\x)-x\x=(\sgn \x)\int_0^x(\sqrt{\x^2-V(s)}-|\x|)ds=A_{\pm}(\x)+R_{\pm}(x,\x).
\end{align*}
Then, 
\begin{align}\label{eq:wavecomparisionAR}
|\pa_x^{\a}\pa_{\x}^{\b}(e^{iR_{\pm}(x,\x)}-1)|\lesssim \jap{x}^{1-\m-|\a|}\jap{\x}^{-1-|\a|}\quad \text{for $(x,\x)\in \supp a_{+}\cap \{\pm \x>0\}$}.
\end{align}

We will show
\begin{align}\label{eq:modwaveusualwave}
W_{+}\g_{\pm}(D_x)=\tilde{W}_{+}\g_{\pm}(D_x)e^{iA_{\pm}(D_x)}.
\end{align}
It suffices to prove $\lim_{t\to \infty}(J_+e^{-iA_{\pm}(D_x)}e^{-itH_0}\g_{\pm}(D_x)f-e^{-itH_0}\g_{\pm}(D_x)f) =0$ for $f\in L^2(\R)$ satisfying $\hat{f}\in C_c^{\infty}(\R\setminus \{0\})$.
Due to the assumption $\m>1$, $J_{+}e^{-iA_{\pm}(D_x)}\g_{\pm}(D_x)$ are pseudodifferential operators with the symbols $e^{iR_{\pm}(x,\x)}a_+(x,\x)\g_{\pm}(\x)$ and
\begin{align*}
J_+e^{-iA_{\pm}(D_x)}e^{-itH_0}\g_{\pm}(D_x)f=a_+(x,D_x)\g_{\pm}(D_x)e^{-itH_0}f+((e^{iR_{\pm}}-1)a_+)(x,D_x)\g_{\pm}(D_x)e^{-itH_0}f.
\end{align*}
The velocity estimate \cite[Lemma 2.24]{I} and \eqref{eq:wavecomparisionAR} with $\m>1$ show that the second term vanishes as $t\to \infty$. Since $a_+$ is an outgoing cut-off function, \cite[(2.60)]{I} yields $a_+(x,D_x)\g_{\pm}(D_x)e^{-itH_0}f-e^{-itH_0}\g_{\pm}(D_x)f\to 0$ as $t\to \infty$. In this way, the identity \eqref{eq:modwaveusualwave} is proved.

When $1<p<\infty$, the identity \eqref{eq:modwaveusualwave} and the H\"ormander-Mikhlin theorem \cite[Theorem 6.2.7]{G1} prove the equivalence of the $L^p$-boundedness of $W_+$ and $\tilde{W}_+$. To deal with the endpoint cases $p=1,\infty$, we need additional arguments.
By \eqref{eq:modwaveusualwave}, $\g=\g_++\g_-$, $\supp\g_+\cap \supp\g_-=\emptyset$, and $\g\chi=\chi$, we obtain
\begin{align*}
W_{+}\chi(D_x)=&\tilde{W}_{+}\chi(D_x)\left(\g(D_x)+(e^{iA_{-}(D_x)}-e^{iA_{+}(D_x)})\g_-(D_x) e^{-iA_+ (D_x)}\right)e^{iA_+(D_x)},\\
\tilde{W}_{+}\chi(D_x)=&W_{+}\chi(D_x)\left(\g(D_x)+(e^{-iA_{-}(D_x)}-e^{-iA_{+}(D_x)})\g_-(D_x) e^{iA_+ (D_x)}\right)e^{-iA_+(D_x)}.
\end{align*}
By the definition of $A_{\pm}$, for $\s_1,\s_2\in \{\pm\}$, we have $|\pa_{\x}^{\a}(e^{\s_1iA_{\s_2}(\x)}-1)|\lesssim \jap{\x}^{-1-|\a|}$ uniformly in $\x^2\geq \l_0^2$. Moreover, $\g(\x)-1\in C_c^{\infty}(\R)$. Then,
\begin{align*}
\g(D_x)(e^{\pm iA_{+}(D_x)}-I),\quad (e^{\pm iA_{-}(D_x)}-e^{\pm iA_{+}(D_x)})\g_-(D_x)e^{\mp iA_+(D_x)},\quad \g(D_x)-I
\end{align*}
are bounded on $L^p$ by \cite[\S7, Lemma 5.2]{Tay}. Moreover, since the identity operator is bounded on $L^p$, we find that $\g(D_x)=I+(\g(D_x)-I)$ is also bounded on $L^p$. Thus,
\begin{align*}
\left(\g(D_x)+(e^{\pm iA_{-}(D_x)}-e^{\pm iA_{+}(D_x)})\g_-(D_x) e^{\mp iA_+(D_x)}\right)e^{\pm iA_+(D_x)}
\end{align*}
are bounded on $L^p$. Therefore, the $L^p$-boundedness of $W_{+}\chi(D_x)$ and $\tilde{W}_{+}\chi(D_x)$ is equivalent.

\noindent$(ii)$ We define $\phi_{\DD,+}(x,\x)=\phi_+(x,\x)-x\x-2^{-1}\x^{-1}M_{\DD}(x)$, 
\begin{align*}
B_{\pm}(\x)=(\sgn \x)\int_0^{\pm\infty}\left(\sqrt{|\x|^2-V(s)}-|\x|-2^{-1}|\x|^{-1}V(s)\right)ds
\end{align*}
 and $E_{\pm}(x,\x):=-(\sgn \x)\int_x^{\pm\infty}(\sqrt{\x^2-V(s)}-|\x|-2^{-1}|\x|^{-1}V(s))ds$ for $\x^2\geq \l_0^2$ so that
\begin{align*}
\phi_{\DD,+}(x,\x)=B_{\pm}(\x)+E_{\pm}(x,\x).
\end{align*}
Set $J_{\DD,+}u(x)= \frac{1}{2\pi}\int_{\R^{2}} e^{i(\phi_{\DD,\pm} (x, \xi) -y\xi)} a_{\pm}(x, \xi) f(y) dyd\x$. Similarly to the proof of \eqref{eq:wavecomparisionAR}, we have
\begin{align*}
W_{+}\g_{\pm}(D_x)=\operatorname*{s-lim}_{t \to  \infty}e^{itH}J_{\DD,+} e^{-itH_0}\g_{\pm}(D_x)e^{iB_{\pm}(D_x)}.
\end{align*}
Moreover, the stationary phase theorem\footnote{In this case, the stationary point is given by $\x=x/2t$. Around this point, $(2\x)^{-1}\int_0^xV(s)ds \cong 2^{-1}\int_{0}^tV(s\x)ds$ on $\supp a_{\pm}$ as $t\to \infty$, which is the idea of the proof for this part.} shows that $J_{\DD,+}e^{-itH_0}f=e^{-itH_0}e^{-\frac{i}{2}\int_0^tV(sD_x)ds}f +o_{L^2}(1)$ as $t\to \infty$ for $f\in L^2(\R)$ satisfying $\hat{f}\in C_c^{\infty}(\R\setminus \{0\})$ and $\operatorname*{s-lim}_{t \to \pm \infty}e^{itH}J_{\DD,+} e^{-itH_0}=W_{\DD,\pm}$. Therefore,
\begin{align}\label{eq:modwaveDollard}
W_{+}\g_{\pm}(D_x)=W_{\DD,+}\g_{\pm}(D_x)e^{iB_{\pm}(D_x)}.
\end{align}
The remaining part of the proof is entirely the same as in that of $(i)$. We omit the detail.

\noindent$(iii)$ The assertion follows from $W_{\pm}\chi(D_x) = W_{\mathrm{Y}, \pm}\chi(D_x)$, which we prove below. We only consider the upper sign. First we construct $\Xi_+$ in the definition of $W_{\mathrm{Y}, +}$. For any $\epsilon \in (0, 1)$ and $(t, x) \in (0, \infty) \times \R$ satisfying $\epsilon \le \left|\frac{x}{2t}\right| \le \frac{1}{\epsilon}$, if $t \ge T_{\epsilon}$ is sufficiently large, there exists unique $\xi (t, x) \in \R$, which enjoys
\begin{align*}
\xi (t, x) = \frac{x}{2t} + \frac{1}{2t} \pa_{\xi} \Phi_+ (x, \xi (t, x)), \quad \left|\xi (t, x) -\frac{x}{2t} \right| < \frac{\epsilon}{10},
\end{align*}
where we set $\Phi_+ (x, \xi) =\phi_+ (x, \xi) -x\xi$. This follows from a standard contraction argument. Then we set $\Xi_+ (t, x) = x\xi (t, x) +\Phi_+ (x, \xi (t, x)) -|\xi (t, x)|^2 t$, which satisfies the Hamilton--Jacobi equation since we have
\begin{align*}
&\pa_t \Xi_+ (t, x) = -|\xi (t, x)|^2, \quad \pa_x \Xi_+ (t, x) = \xi (t, x) + \pa_x \Phi_+ (x, \xi (t, x)), \\
&\pa_t \Xi_+ (t, x) + |\pa_x \Xi_+ (t, x)|^2 +V(x) =2\xi (t, x) \pa_x \Phi_+ (x, \xi (t, x)) +|\pa_x \Phi_+ (x, \xi (t, x))|^2 +V(x) =0,
\end{align*}
where we have used the definition of $\Phi_+$.

Now, for $\hat{f} \in C^{\infty} _c (\{\epsilon \le |\xi| \le 1/\epsilon\})$, we proceed to calculate
\begin{align*}
J_+ e^{-itP_0} f(x)= (2\pi)^{-\frac{1}{2}} \tilde{\chi} \left(\frac{x}{2t}\right) \int_{\R} e^{i\phi_+ (x, \xi)-it|\xi|^2} a_+ (x, \xi) \hat{f} (\xi) d\xi + o_{L^2} (1)
\end{align*}
as $t \to \infty$ by the non-stationary phase method, where $\tilde{\chi} \in C^{\infty} _c (\R \setminus \{0\})$ satisfies $\tilde{\chi} =1$ on $\{\epsilon/2 \le |\xi|\le 2/\epsilon \}$. Now the stationary phase theorem yields
\begin{align}
J_+ e^{-itP_0} f(x)&= e^{i\Xi_+ (t, x)} (2it)^{-\frac{1}{2}} a_+ (x, \xi (t, x)) \hat{f} (\xi (t, x)) + o_{L^2} (1)\notag \\
&= e^{i\Xi_+ (t, x)} (2it)^{-\frac{1}{2}} \bar{\chi} \left(\left|\frac{x}{2t\l_0}\right|\right) \hat{f} \left(\frac{x}{2t}\right) + o_{L^2} (1),\label{eq:J_+expansion}
\end{align}
where we have used $\left|\xi (t, x) - \frac{x}{2t}\right| \lesssim O (t^{-\mu})$. Therefore we obtain
\begin{align*}
W_{+} f = \operatorname*{s-lim}_{t \to \infty} e^{itP} J_+ e^{-itP_0} f= \operatorname*{s-lim}_{t \to \infty} e^{itP} \mathcal{U}_{0, +} (t)\bar{\chi} ( |D_x /\l_0|) f = W_{\mathrm{Y}, +} f,
\end{align*}
which completes the proof.

\noindent$(iv)$ 
The statement follows from $W_{\pm}\chi(D_x) = W_{\mathrm{H}, \pm}\chi(D_x)$, which we prove below. Similarly to the case $(iii)$, we consider the upper sign. First we construct $S_+$ in the definition of $W_{\mathrm{H}, +}$. For $\xi \in \R$ satisfying $\epsilon \le |\xi| \le 1/\epsilon$, $t>1$ and $z \in \R$ satisfying $\epsilon/10 \le |z| \le 10/\epsilon$, we set $\Psi_{+} (t, z, \xi) = -\frac{1}{t} (\Xi_+ (t, 2tz) -2tz\xi)$. Then we have
\begin{align*}
\pa_z \Psi_+ (t, z, \xi) = 2(\xi - (\pa_x \Xi_+) (t, 2tz)), \quad \pa^2 _z \Psi_+ (t, z, \xi) = -4t(\pa^2 _x \Xi_+) (t, 2tz).
\end{align*}
Using $|\xi (t, 2tz) -z| \lesssim t^{-\mu}$ and $|(\pa_x \xi)(t, 2tz) -\frac{1}{2t}| \lesssim t^{-(1+\mu)}$, we obtain
\begin{align*}
\pa_z \Psi_+ (t, z, \xi) =2(\xi -z + O(t^{-\mu})), \quad \pa^2 _z \Psi_+ (t, z, \xi) = -2 + O(t^{-\mu}),
\end{align*}
which implies if $t \ge T_{\epsilon}$ is suffuciently large, $\pa^2 _z \Psi_+ (t, z, \xi) <0$ and $\pa_z \Psi_+ (t, z, \xi) =0$ has a unique solution $z=z(t, \xi)$ satisfying $|z(t, \xi) -\xi| \lesssim t^{-\mu}$. Now we define $S_+ (t, \xi) = t\Psi_+ (t, z(t, \xi), \xi)$, which satisfies the Hamilton--Jacobi equation. In fact, we can use
\begin{align*}
\pa_t S_{+} (t, \xi) = -(\pa_t \Xi_+) (t, 2tz(t, \xi)), \quad \pa_{\xi} S_+ (t, \xi) = 2tz (t, \xi),
\end{align*}
and the fact that $\Xi_+$ satisfies the Hamilton--Jacobi equation in the configuration space. Now, for $\hat{f} \in C^{\infty} _c (\{10\epsilon \le |\xi| \le 1/10\epsilon\})$, we proceed to calculate
\begin{align*}
\mathcal{F} [J_+ e^{-itP_0}f] (\xi) &= \mathcal{F}\left[e^{i\Xi_+ (t, \cdot)} (2it)^{-\frac{1}{2}} \bar{\chi} \left(\left|\frac{\cdot}{2t\l_0}\right|\right) \hat{f} \left(\frac{\cdot}{2t}\right)\right](\xi) + o_{L^2}(1) \\
&=2t(2it)^{-\frac{1}{2}} (2\pi)^{-\frac{1}{2}} \tilde{\chi} (\xi) \int_{\R} e^{-it\Psi_+ (t, z, \xi)} \bar{\chi} \left( \left|\frac{z}{\l_0} \right|\right)\hat{f}(z) dz + o_{L^2}(1),
\end{align*}
where we use \eqref{eq:J_+expansion} and take $\tilde{\chi} \in C^{\infty} _c (\R \setminus \{0\})$ satisfying $\supp \tilde{\chi} \subset \{4\epsilon \le |\xi| \le 1/4\epsilon\}$ and $\tilde{\chi} (\xi) =1$ on $\{6\epsilon \le |\xi| \le 1/6\epsilon\}$, after which we perform non-stationary phase argument. Now the stationary phase theorem yields
\begin{align*}
\mathcal{F} [J_+ e^{-itP_0}f] (\xi) =e^{-iS_+ (t, \xi)} \bar{\chi} \left( \left|\frac{z(t, \xi)}{\l_0} \right|\right)\hat{f}(z(t, \xi)) + o_{L^2} (1) = e^{-iS_+ (t, \xi)} \bar{\chi} \left( \left|\frac{\xi}{\l_0} \right|\right)\hat{f}(\xi)+ o_{L^2} (1),
\end{align*}
which implies
\begin{align*}
W_{+}f=\operatorname*{s-lim}_{t \to \infty} e^{itP} J_+ e^{-itP_0} f= \operatorname*{s-lim}_{t \to \infty} e^{itP}e^{-iS_+ (t, D_x)} \bar{\chi} (|D_x /\l_0|)f = W_{\mathrm{H}, +} f,
\end{align*}
and the proof is colmpleted\footnote{In \cite[\S 4]{DG}, a similar relation is proved though they seem to employ different solutions to the eikonal and the Hamilton--Jacobi equations. If we use H\"ormander's modified wave operators in \cite{DG}, there will appear a Fourier multiplier with an oscillating symbol in the relation to our Isozaki--Kitada modifiers, which causes a difficulty in the high energy regime.}.
%The statement follows from $W_{\pm}\eta(D_x)e^{iC_{\pm}(D_x)} = W_{\mathrm{H}, \pm}\eta(D_x)$, where $C_{\pm}$ are smooth near the support of $\eta$. This is essentially proved in \cite[Theorem 4.9.1]{DG}. Note that the phase function $\Phi^+ _{\mathrm{lr}}$ in the Isozaki--Kitada modifier defined in \cite[(4.9.4)]{DG} is not equal to our phase function $\phi_{+}$ but we have $\Phi^+ _{\mathrm{lr}} = \phi_+ + C_+$ instead due to the condition \cite[(4.9.2)]{DG}.
\end{proof}

\begin{remark}
\noindent$(1)$
If we further assume that $V(x)=V(-x)$ for all $x\in \R$, we have $A_+(\x)=A_-(\x)$ and $B_+(\x)=B_-(\x)$. Therefore, $W_+\chi(D_x)=\tilde{W}_+\chi(D_x)e^{iA_+(D_x)}$ and  $W_{+}\chi(D_x)= W_{\DD,+}\chi(D_x)e^{iB_{+}(D_x)}$.

\noindent$(2)$ Solutions of the eikonal equation $|\pa_x\phi_{\pm}(x,\x)|^2+V(x)=\x^2$ and the Hamilton--Jacobi equations \eqref{eq:HamJac1} and \eqref{eq:HamJac2} are not unique. Therefore, the Isozaki--Kitada, Yafaev's, and H\"ormander modifiers are not uniquely determined either. The proofs of $(iii)$ and $(iv)$ show that if we choose solutions of these three equations compatibly, then the corresponding three modifiers coincide.

\end{remark}

\subsection{The stationary representation of the modified wave operators and the wave matrices}
Following \cite[\S6.1, Appendix A]{DS}, we derive the stationary representation of the modified wave operators $W_{\pm}$ defined in \eqref{eq:IKmodwave}. We define the Fourier restriction/extension operators $\mathcal{F}_{0}(\l),\mathcal{F}_{0}(\l)^*$ in one dimension by
\begin{align*}
\mathcal{F}_{0}(\l)f(x)=\frac{1}{(2\pi)^{\frac{1}{2}}}\begin{pmatrix}
\int_{\mathbb{R}} e^{-i\l x}f(x)dx\\
\int_{\mathbb{R}} e^{i\l x}f(x)dx
\end{pmatrix},\quad (\mathcal{F}_0(\l)^*{\bm c})(x):=\frac{1}{(2\pi)^{\frac{1}{2}}}(e^{i\l x}c_1+e^{-i\l x}c_2)
\end{align*}
for a function $f$ on $\R$ and ${\bm c}={}^t(c_1,c_2)\in\mathbb{C}^2$. We introduce the \textbf{wave matrices} $W_{\pm}(\l)$ of $W_{\pm}$ by
\begin{align}\label{eq:wavematJpm}
W_{\pm}(\l):=J_{\pm}(\l)+iR(\l\mp i0)T_{\pm}(\l),
\end{align}
where $R(\l\mp i0):=\lim_{\e\searrow 0}(P-(\l\mp i\e)^2)^{-1}$ and
\begin{align*}
T_{\pm}:=i(PJ_{\pm}-J_{\pm}P_0),\quad J_{\pm}(\l):=J_{\pm}\mathcal{F}_{0}(\l)^*,\quad T_{\pm}(\l):=T_{\pm}\mathcal{F}_{0}(\l)^*.
\end{align*}
By a simple calculation as in \cite[(5.6)]{DS}, we have $T_{\pm}f(x)=\frac{1}{2\pi}\int_{\R^{2}} e^{i(\phi_{\pm} (x, \xi) -y\xi)} t_{\pm}(x, \xi) f(y) dyd\xi$, where
\begin{align}\label{eq:tpmdef}
t_{\pm}(x,\x):=\left(2\pa_x\phi_{\pm}(x,\x)\pa_x+\pa_x^2\phi_{\pm}(x,\x)\right)a_{\pm}(x,\x)-i\pa_x^2a_{\pm}(x,\x).
\end{align}
Since $\pa_xa_{\pm}(x,\x)$ is compactly supported in $x$, $|\pa_x\phi_{\pm}(x,\x)|\lesssim \jap{\x}$, and $|\pa_x^2\phi_{\pm}(x,\x)|\lesssim \jap{x}^{-1-\mu}$, we have $|t_{\pm}(x,\x)|\lesssim \jap{x}^{-1-\mu}\jap{\x}$, which is used in the next subsection. Mimicking the proof of \cite[Lemma 6.2]{DS} and \cite[(6.8)]{DS}, for $\g\in C_c^{\infty}((0,\infty))$, we obtain the stationary representation
\begin{align}\label{eq:wavestrep}
W_{\pm}\g(P_0)=\int_0^{\infty}\g(\l^2)W_{\pm}(\l)\mathcal{F}_0(\l)d\l.
\end{align}

\subsection{Jost solutions and their asymptotic properties}\label{subsection:Jost}
In view of \eqref{eq:wavematJpm}, the outgoing part $J_+(\l)$ of $W_{+}(\l)$ is explicitly given in \eqref{eq:Jpmdef}; however, the incoming part $iR(\l\mp i0)T_{\pm}(\l)$ is not. To give a more explicit expression for the integral kernel of $W_{\pm}$, we use the Jost functions constructed in \cite{HT} and their asymptotic properties established there. In this subsection, we summarize the required properties of the Jost functions.

Let $u_{\pm}(x,\l)$ be the Jost solutions constructed in\footnote{Although \cite[Theorem 3.11]{HT} is stated when $-V(x)\gtrsim \jap{x}^{-\mu}$ and $0<\m<2$, all the results still hold for potentials satisfying \eqref{26641742} with $\m>0$ whenever we restrict the range of the parameter $\l$ to $\l\geq \l_0$.} \cite[Theorem 3.11]{HT} for equation 
\begin{align}\label{eq:stationarySch}
(-\pa_x^2+V(x)-\l^2)u(x,\l)=0\quad (x\in \R,\,\, |\l|\geq \l_0)
\end{align}
and 
\begin{align}\label{eq:upasymp}
u_{\pm} (x, \lambda) =\frac{1}{(\l^2-V(x))^{\frac{1}{4}}} \left(\tilde{a}_{\pm, +} (x, \lambda) e^{iy(x,\l)} + \tilde{a}_{\pm, -} (x, \lambda) e^{-iy(x,\l)}\right),
\end{align}
where $y(x,\l)=\int_0^x\sqrt{\l^2-V(s)}ds$ is defined in \eqref{eq:yphidef} and $\tilde{a}_{\s_1,\s_2}$ are the functions satisfying the following properties: We define the Wronskian by
\begin{align}\label{eq:defWronskian}
\mathrm{Wr}(\lambda):= u_{+} (x, \lambda) \partial_{x} u_{-} (x, \lambda) - \partial_{x} u_{+} (x, \lambda) u_{-} (x, \lambda).
\end{align}

\begin{lemma}\label{lem:apropertyold}
The coefficients $\tilde{a}_{\s_1,\s_2}$ for $\s_1,\s_2\in \{\pm\}$ satisfy the following properties:

\noindent$(i)$ The limits $\displaystyle \tilde{a}_{\s_1,\s_2}(\pm\infty,\l):=\lim_{x\to \pm\infty}\tilde{a}_{\s_1,\s_2}(x,\l)$ exist uniformly in $|\l|\geq \l_0$ and
\begin{align*}
\tilde{a}_{+,+}(\infty,\l)=\tilde{a}_{-,-}(-\infty,\l)=1,\quad \tilde{a}_{+,-}(\infty,\l)=\tilde{a}_{-,+}(-\infty,\l)=0.
\end{align*}

\noindent$(ii)$ For $\a=0,1,2$,
\begin{align*}
|\partial^{\a} _{\lambda} \tilde{a}_{\s_1,\s_2} (x, \lambda)| \lesssim |\lambda|^{-\a}, \quad |\pa_x\tilde{a}_{\s_1,\s_2}(x,\l)|=o(1)\quad (|x|\to \infty)
\end{align*}
uniformly in $|\l|\geq \l_0$.

\noindent$(iii)$ We have $\mathrm{Wr}(\l)=-2i\tilde{a}_{+,+}(-\infty,\l)=-2i \tilde{a}_{-, -} (\infty, \lambda)$ and for $\a\in\mathbb{N}_0$,
\begin{align*}
|\mathrm{Wr}(\lambda)| \sim 1, \quad \left| \left(\frac{d}{d\lambda} \right)^{\alpha} \mathrm{Wr}(\lambda) \right| \lesssim |\lambda|^{-\alpha}
\end{align*}
uniformly in $|\l|\geq \l_0$. In particular,
\begin{align}\label{eq:apmelliptic}
\tilde{a}_{+,+}(x,\l)=\left\{\begin{aligned}
&1+o(1) &&\text{as $x\to \infty$} \\
&\frac{1}{2i}\mathrm{Wr}(\l)+o(1) &&\text{as $x\to -\infty$}
\end{aligned}\right.
,\quad \tilde{a}_{-,-}(x,\l)=\left\{\begin{aligned}
&\frac{1}{2i}\mathrm{Wr}(\l)+o(1) &&\text{as $x\to \infty$} \\
&1+o(1) &&\text{as $x\to -\infty$}
\end{aligned}\right.
\end{align}
uniformly in $|\l|\geq \l_0$.
\end{lemma}

\begin{proof}
$(i)$ and $(ii)$ are stated in \cite[Theorem 3.11 (ii)]{HT}. The estimates for $\tilde{a}_{\s_1,\s_2}(\pm\infty,\l)$ and the identity for $\mathrm{Wr}(\l)$ are given in the proof of \cite[Theorem 3.11 (iv)]{HT}. The estimates for $\mathrm{Wr}(\l)$ is stated in \cite[Theorem 3.11 (iv)]{HT}. The asymptotics \eqref{eq:apmelliptic} follow from the uniformness of the limits in $(i)$ and $|\tilde{a}_{\s_1,\s_2}(\pm\infty,\l)|\sim 1$.
\end{proof}

%Before giving a proof, we collect asymptotic properties of the amplitude in the Jost functions constructed in \cite{HT}.

The following lemma contains a refinement of \cite[(3.29), (3.31) and Theorem 3.11 $(iv)$]{HT} in the high energy regime.
%We write $w(\lambda):=w\{u_{-}, \overline{u_{+}}\} (\lambda)$ and $W(\lambda):= \mathrm{Wr} (\lambda)$ for the sake of simplicity. 

\begin{lemma}\label{264112228}
We have
\begin{align*}
|\tilde{a}_{+, +} (x, \lambda)-1|+|\tilde{a}_{-, -} (x, \lambda)-1|+|\tilde{a}_{+, -} (x, \lambda)|+|\tilde{a}_{-, -} (x, \lambda)|+|\mathrm{Wr}(\lambda)+2i| \lesssim |\lambda|^{-3}
\end{align*}
uniformly in $x\in \R$ and $|\lambda| \geq \l_0$. 

%\noindent$(i)$ We have $|\tilde{a}_{+, -} (x, \lambda)|+|\tilde{a}_{-, -} (x, \lambda)| \lesssim |\lambda|^{-3}$ uniformly in $x\in \R$ and $|\lambda| >\frac{c}{2}$.
%
%\noindent$(ii)$ We have $|\tilde{a}_{+, +} (x, \lambda)-1|+|\tilde{a}_{-, -} (x, \lambda)-1| \lesssim |\lambda|^{-3}$ uniformly in $x\in \R$ and $|\lambda| >\frac{c}{2}$.

%\noindent$(iii)$ We have $|w(\lambda)| \lesssim |\lambda|^{-3}$ uniformly in $|\lambda| > \frac{c}{2}$.
%
%\noindent$(iv)$ We have $|W(\lambda)+2i| \lesssim |\lambda|^{-3}$ uniformly in $|\lambda| > \frac{c}{2}$.
\end{lemma}

\begin{proof}
We briefly recall the construction of $\tilde{a}_{\s_1,\s_2}$ form \cite[\S 3]{HT}. We define
\begin{align*}
U(x,\lambda)=\frac{V''(x)}{4(\lambda^2-V(x))^2}-\frac{5(V'(x))^2}{16(\lambda^2-V(x))^{3}} ,\quad W(y,\lambda)=U(x(y,\l),\l)
\end{align*}
as in \cite[$(3.34)$]{HT}, where $x(y,\l)$ is the inverse function of $x\mapsto y(x,\l)$. Clearly, we see that $|U(x, \lambda)| \lesssim \jap{x}^{-2-\mu} |\lambda|^{-4}$ for $|\lambda|\geq \l_0$ and hence $|W(y,\l)|\lesssim \jap{\l^{-1}y}^{-2-\mu} |\lambda|^{-4}$ by $y(x, \lambda) \sim |\lambda|x$ for $\l\geq \l_0$.
The vector valued functions  $\textit{\textbf{w}}_{\pm}(y,\l) ={}^t(w_{1, \pm}(y,\l), w_{2, \pm}(y,\l))$ satisfy $\lim_{y\to \pm\infty}\textit{\textbf{w}}_{\pm} (y, \lambda)={}^t(1,\pm i)$ and
\begin{align}\label{eq:vectorweq}
\partial_{y} \textit{\textbf{w}}_{\pm} (y, \lambda)=B(y,\l) \textit{\textbf{w}}_{\pm} (y, \lambda),\quad
B(y,\l):=-W(y,\l)\begin{pmatrix}
\sin s \cos s & \sin ^2 s \\
-\cos ^2 s & -\sin s \cos s
\end{pmatrix}
\end{align}
as is seen in \cite[Theorem 3.6 and around (3.17)]{HT}. By using \cite[Theorem 3.6]{HT} (or simply the decay of $W(y,\l)$ with respect to $y$), we find that all the limits $\textit{\textbf{w}}_{\s_2} (\s_1\infty, \lambda):=\lim_{y\to \s_1\infty}\textit{\textbf{w}}_{\s_2} (y, \lambda)$ for $\s_1,\s_2\in \{\pm\}$ exist. By \eqref{eq:vectorweq} and $|\textit{\textbf{w}}_{\pm} (y, \lambda)|\lesssim 1$, we have
\begin{align*}
|\textit{\textbf{w}}_{\s_2} (y, \lambda)-\textit{\textbf{w}}_{\s_2} (\infty, \lambda)|\lesssim\left|\int_{\infty}^yB(s,\l)\textit{\textbf{w}}_{\pm} (y, \lambda) ds\right|\lesssim |\l|^{-4}\left|\int_{\infty}^y\jap{\l^{-1}s}^{-2-\mu}ds\right|\lesssim |\l|^{-3}
\end{align*}
uniformly in $y\in \R$ and $\l\geq \l_0$. Taking $y\to -\infty$, we obtain $|\textit{\textbf{w}}_{\s_2} (-\infty, \lambda)-\textit{\textbf{w}}_{\s_2} (\infty, \lambda)|\lesssim |\l|^{-3}$ and hence
\begin{align*}
{}^t(1,i)=\textit{\textbf{w}}_{+} (\infty, \lambda)=\textit{\textbf{w}}_{+} (-\infty, \lambda)+O(|\l|^{-3}),\,\, {}^t(1,-i)=\textit{\textbf{w}}_{-} (-\infty, \lambda)=\textit{\textbf{w}}_{-} (\infty, \lambda)+O(|\l|^{-3})
\end{align*}
According to \cite[around (3.18) and Proof of Theorem 3.11 (i)]{HT}, we define
\begin{align*}
a_{\pm, +} (y, \lambda) = \frac{w_{1, \pm} (y, \lambda) -i w_{2, \pm} (y, \lambda)}{2} \quad \textrm{and} \quad a_{\pm, -} (y, \lambda) = \frac{w_{1, \pm} (y, \lambda) +i w_{2, \pm} (y, \lambda)}{2}
\end{align*}
and $\tilde{a}_{\s_1,\s_2} (x, \lambda)=a_{\s_1,\s_2}(y(x,\l),\l)$.
From the above, we see
\begin{align*}
&|a_{\s_1, \s_2} (y, \l)-a_{\s_1,\s_2}(\pm\infty,\l)|\lesssim |\l|^{-3},\\
&|a_{+, +} (x, \lambda)-1|+|a_{-, -} (x, \lambda)-1|+|a_{+, -} (x, \lambda)|+|a_{-, -} (x, \lambda)| \lesssim |\lambda|^{-3}.
\end{align*}
The claim follows from $\tilde{a}_{\s_1,\s_2}(\pm\infty,\l) =a_{\s_1,\s_2}(\pm\infty,\l)$ and $\mathrm{Wr}(\lambda) = -2ia_{-, -} (\infty, \lambda)$, where the last identity follows from the proof of \cite[Theorem 3.11 (iv)]{HT}.
\end{proof}

\subsection{Oscillatory integral representation of the modified wave operators}
Using the stationary representation \eqref{eq:wavestrep} of $W_{\pm}$, we can express their integral kernels in terms of the amplitudes $\tilde{a}_{\s_1,\s_2}$ appearing in the asymptotic expansion \eqref{eq:upasymp} of the Jost solutions. In particular, we obtain their oscillatory integral representations as follows. For the sake of brevity, we give the expression of $W_+$ only.

\begin{proposition}\label{prop:waveexpression1}
Let $\g\in C^{\infty}(\R)$ satisfy $\g\in C_c^{\infty}((\sqrt{2}\l_0,\infty))$ or $\g (\l) =1$ for $\l \gg 1$ with $\supp \g\subset (\sqrt{2}\l_0,\infty)$. Then we have
\begin{align}\label{eq:waveopdecom}
W_{+}\g(P_0)=\sum_{\s_1,\s_2\in\{\pm\}}W_{+,\s_1,\s_2,\g},\,\, W_{+,\s_1,\s_2,\g}(x,x')=\frac{1}{2\pi}\int_0^{\infty} b_{\s_1,\s_2,\g}(x,\l)e^{i(\s_2y(x,\l) -\s_1\l x')}d\l,
\end{align}
where $b_{\s_1,\s_2,\g}$ satisfies $\supp b_{\s_1,\s_2,\g}(x,\l)\subset \supp \g(\l^2)$ and for $\a=0,1,2$ and $\s_1,\s_2\in \{\pm\}$,
\begin{align}
&|\pa_{\l}^{\a}b_{\s_1,\s_2,\g}(x,\l)|\lesssim \l^{-\a}\quad \text{for $x\in \R$},\quad |b_{\pm,\pm,\g}(x,\l)|\gtrsim |\g(\l^2)|\quad \text{for $|x|\gg 1$}\label{eq:b_+jsymbest}\\
&b_{+,+,\g}(x,\l)= (1+o(1))\g(\l^2)\quad \text{as $x\to \infty$}\label{eq:b++--asymp}
\end{align}
uniformly in $\l\geq \sqrt{2}\l_0$ and
\begin{align}
&\left|b_{+,+,\g}(x,\l)-\g(\l^2)\right|+\left|b_{-,-,\g}(x,\l)-\g(\l^2)\right|+|b_{+,-,\g}(x,\l)|+|b_{-,+,\g}(x,\l)| \lesssim \jap{\l}^{-3}\label{eq:bpmlambdalarge}
\end{align}
uniformly in $x\in \R$ and $\l\geq \sqrt{2}\l_0$.

\end{proposition}

\begin{remark}
Similarly, we have $b_{-,-,\g}(x,\l)=(1+o(1)) \g(\l^2)$ as $x\to -\infty$ uniformly in $\l\geq \sqrt{2}\l_0$. 
\end{remark}

\begin{proof}
First, we show that the wave matrix $W_+(\l)$ satisfies
\begin{align}\label{eq:wavematformula}
(W_+(\l){\bm c})(x)=\frac{2i\l^{\frac{1}{2}}}{(2\pi)^{\frac{1}{2}}\overline{\mathrm{Wr}(\l)}}(\overline{u_-(x,\l)}c_1+\overline{u_+(x,\l)}c_2)
\end{align}
for ${\bm c}={}^t(c_1,c_2)\in \mathbb{C}$ and $\l\geq \sqrt{2}\l_0$, where $\mathrm{Wr}(\lambda)$ is defined in \eqref{eq:defWronskian}. Since the both side of \eqref{eq:wavematformula} satisfies the Schr\"odinger equation \eqref{eq:stationarySch}, it suffices to prove
\begin{align*}
((\sgn x)\pa_x+i\l)(\text{the LHS of \eqref{eq:wavematformula}})=((\sgn x)\pa_x+i\l)(\text{the RHS of \eqref{eq:wavematformula}})+o(1)\quad \text{as $|x|\to \infty$}
\end{align*}
by Sommerfeld's radiation condition \cite[p.407]{Y}.
By the definitions of $J_{+}$, $T_+$ and $\mathcal{F}_0^*(\l)$, and the Fourier inverse formula, we have 
\begin{align*}
(J_{+}(\l){\bm c})(x)=&\frac{1}{(2\pi)^{\frac{1}{2}}}(e^{i\phi_+(x,\l)}c_1a_{+}(x,\l)+e^{i\phi_+(x,-\l)}c_2a_{+}(x,-\l)),\\
(T_{+}(\l){\bm c})(x)=&\frac{1}{(2\pi)^{\frac{1}{2}}}(e^{i\phi_+(x,\l)}c_1t_{+}(x,\l)+e^{i\phi_+(x,-\l)}c_2t_{+}(x,-\l)).
\end{align*}
In particular, $(J_{+}(\l){\bm c})(x)=\frac{1}{(2\pi)^{\frac{1}{2}}}e^{iy(x,\l)}c_1$ for $x\gg 1$ and $\l\geq \sqrt{2}\l_0$, and  $(J_{+}(\l){\bm c})(x)=\frac{1}{(2\pi)^{\frac{1}{2}}}e^{-iy(x,\l)}c_2$ for $x\ll -1$ and $\l\geq \sqrt{2}\l_0$. On the other hand, by the estimate of $t_+$ remarked just after \eqref{eq:tpmdef}, $|(T_{+}(\l){\bm c})(x)|\lesssim \jap{x}^{-1-\m}\jap{\l}\in \jap{x}^{-\frac{1}{2}-\frac{\m}{2}}L^{2}(\R)$. Therefore, $R(\l-i0)(T_{+}(\l){\bm c})\in H_{\mathrm{loc}}^{2}(\R)$ is well-defined and $((\sgn x)\pa_x+i\l)R(\l-i0)(T_{+}(\l){\bm c})(x)=o(1)$ as $|x|\to \infty$ by Sommerfeld's radiation condition \cite[p.407]{Y}. Thus, we obtain
\begin{align*}
((\sgn x)\pa_x+i\l)(\text{the LHS of \eqref{eq:wavematformula}})=
\left\{\begin{aligned}
&\frac{2i\l}{(2\pi)^{\frac{1}{2}}}e^{iy(x,\l)}c_1+o(1)&&(x\to \infty) \\
&\frac{2i\l}{(2\pi)^{\frac{1}{2}}}e^{-iy(x,\l)}c_2+o(1)&&(x\to -\infty).
\end{aligned}\right.
\end{align*}
On the other hand, Lemma \ref{lem:apropertyold} yields
\begin{align*}
&((\sgn x)\pa_x+i\l)\overline{u_{+}(x,\l)}=\left\{\begin{aligned}
&o(1)&&(x\to \infty)\\
&\l^{\frac{1}{2}}\overline{\mathrm{Wr}(\l)}e^{-iy(x,\l)}+o(1)&&(x\to -\infty)
\end{aligned}\right.,\\
&((\sgn x)\pa_x+i\l)\overline{u_{-}(x,\l)}=\left\{\begin{aligned}
&\l^{\frac{1}{2}}\overline{\mathrm{Wr}(\l)}e^{iy(x,\l)}+o(1)&&(x\to \infty)\\
&o(1)&&(x\to -\infty)
\end{aligned}\right.,
\end{align*}
where we have used $\mathrm{Wr}(\l)=-2i\tilde{a}_{+,+}(-\infty,\l)=-2i \tilde{a}_{-, -} (\infty, \lambda)$ as is stated in Lemmas \ref{lem:apropertyold} $(iii)$. Thus, we conclude that $((\sgn x)\pa_x+i\l)(\text{the RHS of \eqref{eq:wavematformula}})$ is equal to $((\sgn x)\pa_x+i\l)(\text{the LHS of \eqref{eq:wavematformula}}) + o(1)$, which proves \eqref{eq:wavematformula}.
 
By \eqref{eq:wavematformula}, the integral kernels of $W_{\pm}(\l)\mathcal{F}_0(\l)$ can be written as
\begin{align}\label{eq:waverepintegrant}
\left(W_{\pm}(\l)\mathcal{F}_0(\l)\right)(x,x')=
\frac{2i\l^{\frac{1}{2}}}{2\pi\overline{\mathrm{Wr}(\l)}}(\overline{u_-(x,\l)}e^{-i\l x'}+\overline{u_+(x,\l)}e^{i\l x'})
\end{align}
for $\l\geq \sqrt{2}\l_0$ in terms of the Jost solutions introduced in the last subsection. Now we define
\begin{align*}
&b_{\s_1,\s_2,\g}(x,\l)=\frac{2i\l^{\frac{1}{2}}}{\overline{\mathrm{Wr}(\l)}(\l^2-V(x))^{\frac{1}{4}}}\overline{\tilde{a}_{\sigma_1',\sigma_2'}(x,\l)}\g(\l^2),
\end{align*}
where $\s_j'=-\s_j$. Now, the results follow from\footnote{Although \eqref{eq:defWronskian} is stated only for compactly supported functions, this formula remains valid for our present function $\g$ by virtue of the standard theory in oscillatory integrals.} \eqref{eq:defWronskian}, and Lemmas \ref{lem:apropertyold} and \ref{264112228}.
\end{proof}

\begin{remark}
Similarly, we have an analogous expression for $W_-(\l)$ by using the identity
\begin{align*}
(W_-(\l){\bm c})(x)=\frac{-2i\l^{\frac{1}{2}}}{(2\pi)^{\frac{1}{2}}\mathrm{Wr}(\l)}(u_+(x,\l)c_1+u_-(x,\l)c_2).
\end{align*}
\end{remark}

\section{The modified wave operators in the middle energy regime}\label{Section:mid}

Throughout this section, we suppose that $V$ satisfies \eqref{26641742}. Here, we divide the proof of Theorem \ref{26661518} into several parts. We recall $M_{\DD} (x):= \frac{1}{2} \int_{0}^{x} V(s)ds$. We focus on $W_+$ only and the other case is similarly handled.

\subsection{The $L^p$-boundedness in the high energy regime}
Let $\psi \in C^{\infty} (\R; [0, 1])$ be such that $\psi (\l) =1$ for $\l\gg 1$ and $\supp \psi \subset (2\l^2 _0, \infty)$.

\begin{thm}\label{264112215}
Let $V \in C^{\infty} (\R; \R)$ satisfy \eqref{26641742} and $M_{\DD}\in L^{\infty}(\R)$. Then the high energy wave operators $W_{+} \psi (P_0)$ are bounded on $L^p (\R)$ for any $p \in [1, \infty]$.
\end{thm}

Before giving a proof, we state the following technical lemma.

\begin{lemma}\label{26542109}
Under the assumptions of Theorem \ref{264112215}, we have $V \in L^2 (\R)$.
\end{lemma}

\begin{proof}
Integrating by parts with $F(x):= \int_{0}^{x} V(s) ds$, we obtain
\begin{align*}
\int_{-R}^{R} V(s)^2 ds = \int_{-R}^{R} F'(s) V(s) ds = F(R)V(R)-F(-R)V(-R) - \int_{-R}^{R} F(s) V'(s) ds
\end{align*}
for $R >0$. Therefore it holds that
\begin{align*}
\|V\|^2 _{L^2 (-R, R)} \lesssim \|F\|_{L^{\infty} (\R)} \|V\|_{L^{\infty} (\R)} + \|F\|_{L^{\infty} (\R)} \int_{-R}^{R} \jap{s}^{-1-\mu} ds
\end{align*}
and $R \to \infty$ gives $V \in L^2 (\R)$.
\end{proof}

\begin{proof}[Proof of Theorem \ref{264112215}] We employ the expression of $W_+\g(P_0)$ given in \eqref{eq:waveopdecom}. Changing the variable $\l\mapsto -\l$, we have
\begin{align*}
W_{+,-,\s_2,\g}(x,x')=\frac{1}{2\pi}\int_{-\infty}^{0} b_{-,\s_2,\g}(x,-\l)e^{i(\s_2y(x,\l) -\l x')}d\l.
\end{align*}
Using this and $\d(x-x')=\frac{1}{2\pi}\int_{\R}e^{i(x-x')\l}d\l$, we rewrite \eqref{eq:waveopdecom} as
\begin{align}
\left(W_{+}\g(P_0)\right)(x,x')&=\d(x-x')+\frac{1}{2\pi}\int_{\R} (f_{1} (x, \lambda)-1) e^{i\lambda (x-x')} d\lambda +\frac{1}{2\pi}  \int_{\R} f_{2} (x, \lambda) e^{-i\lambda (x+x')} d\lambda \notag \\
&=:\d(x-x')+\mathrm{I}_{x, x'}+\mathrm{II}_{x, x'},\label{eq:W+intdecom}
\end{align}
where we set $g(x,\l)=y(x,\l)-|\l|x$ and  
\begin{align*}
f_1 (x, \lambda) = \left\{\begin{aligned}
&b_{+, +,\g} (x, \lambda) e^{ig(x, \lambda)}&&\lambda >0, \\
&b_{-, -,\g} (x, -\lambda) e^{-ig(x, \lambda)}&&\lambda <0,
\end{aligned}\right.,\quad f_2 (x, \lambda) = \left\{\begin{aligned}
&b_{+, -,\g} (x, \lambda) e^{-ig(x,\l)}&&\lambda >0, \\
&b_{-, +,\g} (x, -\lambda) e^{ig(x, \lambda)}&&\lambda <0.
\end{aligned}\right.
\end{align*}

Next, we claim that for $\a=0,1,2$ and $j=1,2$, we have
\begin{align}\label{eq:f_jest}
|\pa_{\l}^{\a}f_j(x,\l)|\lesssim \jap{\l}^{-\a},\,\, \left|f_1(x,\l)-1-\frac{2i}{\lambda}\g(\l^2) M_{\DD}(x)\right|\lesssim \jap{\l}^{-2},\,\, |f_2(x,\l)|\lesssim \jap{\l}^{-3}
\end{align}
uniformly in $x\in \R$ and $\l\in \R\setminus \{0\}$, where we recall $M_{\DD}(x)=\frac{1}{2}\int_{0}^{x} V(s) ds$. Thanks to Taylor's theorem, for $\a\in\mathbb{N}$, we have
\begin{align*}
\left|\pa_{\l}^{\a}R(x,\l)\right|\lesssim |\l|^{-3-\a} \int_{0}^{x}|V(s)|^2ds\lesssim |\l|^{-3-\a},\quad \text{where $R(x,\l):=g(x, \lambda) - |\l|^{-1}M_{\DD}(x)$}.
\end{align*}
Here, we have used $V\in L^2(\R)$, which was proved in Lemma \ref{26542109}. By the assumption $M_{\DD}\in L^{\infty}(\R)$, we have $|g(x,\l)|\lesssim |\l|^{-1}$ and $|\pa_{\l}^{\a}(e^{\pm ig(x,\l)}-1)|\lesssim |\l|^{-1-\a}$. Combining this with \eqref{eq:b_+jsymbest} and the property of $\g$, we obtain the first estimates in \eqref{eq:f_jest}. The third estimate in \eqref{eq:f_jest} is a direct consequence of \eqref{eq:bpmlambdalarge}. The second estimate can be proved by \eqref{eq:bpmlambdalarge} and $|e^{\pm ig(x,\l)}-1\mp ig(x,\l)|\lesssim |g(x,\l)|^2\lesssim |\l|^{-2}$, where the last inequality follows from Taylor's theorem.

Thirdly, we show
\begin{align}\label{eq:I,IIest}
|\mathrm{I}_{x, x'}|\lesssim \jap{x-x'}^{-2},\quad |\mathrm{II}_{x, x'}|\lesssim \jap{x+x'}^{-2}.
\end{align}
Based on the first estimates in \eqref{eq:f_jest}, integration by parts twice yields $|\mathrm{I}_{x, x'}|\lesssim |x-x'|^{-2}$ and $|\mathrm{II}_{x, x'}|\lesssim |x+x'|^{-2}$. On the other hand, by the third estimate of \eqref{eq:f_jest} and the integrability of $\jap{\l}^{-3}$, we obtain $|\mathrm{II}_{x, x'}|\lesssim 1$ and hence $|\mathrm{II}_{x, x'}|\lesssim \jap{x+x'}^{-2}$. Thus, it suffices to prove $|\mathrm{I}_{x, x'}|\lesssim 1$. Similarly to the above discussion, the second estimate of \eqref{eq:f_jest} implies
\begin{align*}
\left|\mathrm{I}_{x, x'}-2iM_{\DD}(x) \int_{\R}\frac{\g(\l^2)}{\lambda} e^{i\lambda (x-x')} d\lambda\right|\lesssim 1.
\end{align*}
Finally, we see that $\int_{\R}\frac{\g(\l^2)}{\lambda} e^{i\lambda (x-x')} d\lambda=2i\int_{0}^{\infty}\frac{\g(\l^2)\sin((x-x')\l)}{\l}d\l$, which is bounded in $x,x'\in \R$. By using $M_{\DD}\in L^{\infty}(\R)$ again, we obtain $|\mathrm{I}_{x, x'}|\lesssim 1$.

By \eqref{eq:I,IIest} and  Young's inequality, we find that the operators which have integral kernels $\mathrm{I}_{x, x'}$ and $\mathrm{II}_{x, x'}$ respectively are bounded on $L^p$. On the other hands, the operator which has the integral kernel $\d(x-x')$ is the identity operator that is clearly bounded on $L^p$ as well. By virtue of the decomposition \eqref{eq:W+intdecom}, we conclude that $W_+\g(P_0)$ is bounded on $L^p$. 
\end{proof}

\begin{remark}\label{rem:CalZyg}
For the boundedness of $W_+\g(P_0)$ on $L^p$ for $1<p<\infty$, the second and the third estimates in \eqref{eq:f_jest} are not needed. Indeed, the first estimate in \eqref{eq:f_jest} and the estimates of $x$-derivatives of $\tilde{a}$ (provided in \cite{HT}) show that the linear operators $T_1$ and $T_2$ that have integral kernels $\frac{1}{2\pi}\int_{\R}f_1(x,\l)e^{i(x-x')\l}d\l$ and $\frac{1}{2\pi}\int_{\R}f_2(x,\l)e^{-i(x-x')\l}d\l$ respectively are Calder\'on--Zygmund operators \cite[Definition 4.1.8]{G2}, which are bounded on $L^p$ by \cite[Theorem 4.2.2]{G2}. Set $(Ru)(x):=u(-x)$ which is clearly bounded on $L^p$. By \eqref{eq:W+intdecom}, we can write $W_+\g(P_0)=T_1+T_2\circ R$, which is bounded on $L^p$ as well. The most difficult part of the above proof is to show that $W_+\g(P_0)$ is bounded on both $L^1$ and $L^{\infty}$. The second and the third estimates in \eqref{eq:f_jest} prevent $W_+\g(P_0)$ from being an operator of Hilbert transform type.

\end{remark}

\subsection{Unboundedness for the case $p <2$}\label{26531322}
Let
$\psi \in C^{\infty} _c (\R; \R)$ be such that $\supp \psi \subset (2\l^2 _0, \infty)$ and not identically zero. We take $\lambda_1 \in (\sqrt{2}\l_0, \infty)$ and $\psi \in C^{\infty} _c ((0, \infty); [0, 1])$ such that $\psi (\lambda_1^2) \neq 0$. Then we have

\begin{thm}\label{264131227}
Suppose that $p \in [1, 2]$ and $M_{\DD} \notin L^{\infty}(\mathbb{R})$. The middle energy wave operator $W_+ \psi (P_0)$ is bounded on $L^p (\R)$ if and only if $p=2$.
\end{thm}

We may assume $M_{\DD} \notin L^{\infty}(\mathbb{R}_{>0})$. We observe that for $\lambda \in \supp \psi$ and $x \ge 1$, 
\begin{align}\label{eq:phasedecom}
y(x, \lambda) = \lambda \int_{0}^{x} \left(1-\frac{V(s)}{\lambda^2}\right)^{\frac{1}{2}} ds =x\l-M_{\DD}(x)\l^{-1} +R (x, \lambda)
\end{align}
holds by Taylor's theorem, where $R$ satisfies
\begin{align}\label{eq:phaserembound}
|\partial^{\alpha} _{\lambda} R (x, \lambda)| \lesssim \frac{1}{\lambda^{3+\alpha}} \int_{0}^{x} |V(s)|^{2} ds,\quad |\partial^{\alpha} _{\lambda} (R (x, \l)-R(x',\l))|\lesssim \frac{1}{\lambda^{3+\alpha}} \int_{x'}^{x} |V(s)|^{2} ds
\end{align}
for $x\geq x'\geq 1$.
Before proceeding to the proof, we need the following technical lemma, which enables us to regard $R(x,\l)$ as a remainder term of $M_{\DD}$.

\begin{lemma}\label{2658032}
Suppose that $M_{\DD} \notin L^{\infty} (\R_{>0})$ holds. Then there exists a sequence $\{x_n\} \subset \R_{>0}$ such that $x_n \le x_{n+1}$ for all $n \in \N$, $x_n \to \infty$ as $n \to \infty$ and
\begin{align}
|M_{\DD} (x_n)| \to \infty, \quad  |M_{\DD} (x_n)|^{-1}\int_{0}^{x_n} V(s)^2 ds \to 0 \quad (n \to \infty). \label{2658814}
\end{align}
\end{lemma}

\begin{proof}
For each $n \in \N$, we set $x_n \in \R_{>0}$ as $x_n \in [0, n]$, $|M_{\DD} (x_n)| = \max_{x \in [0, n]} |M_{\DD} (x)|$ and $x_n \le x_{n+1}$. Since we assume $M_{\DD} \notin L^{\infty} (\R_{>0})$, we have $x_n \to \infty$ and $|M_{\DD} (x_n)| \to \infty$. %If $V\in L^2(\R_{>0})$, then the claim is trivial. Therefore, we may assume $V\notin L^2(\R_{>0})$

Let $\e>0$. We take $a>0$ so that $\int_{a}^{\infty}|V'(s)|ds<\frac{\e}{8}$ and then choose $N\in\mathbb{N}$ such that $|V(x_n)|<\frac{\e}{8}$, and $|M_{\DD}(x_n)|^{-1}|M_{\DD} (a)V(a)|<\frac{\e}{8}$, and $|M_{\DD}(x_n)|^{-1}\int_0^aV(s)^2ds<\frac{\e}{4}$ for $n\geq N$. The integration by parts shows
\begin{align}\label{eq:VL^2id}
\frac{1}{2}\int_{a}^{x_n} V(s)^2 ds = M_{\DD} (x_n) V(x_n) -M_{\DD} (a)V(a) -\int_{a}^{x_n} M_{\DD} (s)V' (s)ds.
\end{align}
We note $\left|\int_{a}^{x_n}M_{\DD}(s)V'(s)ds\right|\leq |M_{\DD}(x_n)|\int_a^{\infty}|V'(s)|ds$ by $x_n\leq n$ and the choice of $x_n$.
Dividing the both side of \eqref{eq:VL^2id} by $M_{\DD}(x_n)$, we have $\frac{1}{2}|M_{\DD}(x_n)|^{-1}|\left|\int_{a}^{x_n} V(s)^2 ds\right|<\frac{3}{8}\e$ for $n\geq N$. Using this and $|M_{\DD}(x_n)|^{-1}\int_0^aV(s)^2ds<\frac{\e}{4}$, we obtain $|M_{\DD}(x_n)|^{-1}|\left|\int_{0}^{x_n} V(s)^2 ds\right|<\e$ for $n\geq N$.
\end{proof}

Let $x_n$ be the sequence in Lemma \ref{2658032}. We define
\begin{align}\label{eq:u_ndef}
u_n(x):=\chi(x-x_n)\quad (n=1,2,\hdots),
\end{align}
where $\chi\in \mathcal{S}(\mathbb{R})$ such that $\chi\neq 0$ and $\hat{\chi}(\l)=1$ for $\l\in \supp \g$ and $\supp \hat{\chi}(\l)\subset (0,\infty)$. 
Firstly we give a uniform $L^p$ bound for $W_{+, +, -, \psi} u_n$.

\begin{lemma}\label{264131246}
For sufficiently large $R_0$, we have $\|W_{+, +, -, \psi} u_n\|_{L^p (\R_{\ge R_0})} \lesssim 1$ uniformly in $n \in \N$.
\end{lemma}

\begin{proof}
We rewrite $W_{+, +, -, \psi} u_n$ as
\begin{align*}
W_{+, +, -, \psi} u_n (x) = \frac{1}{2\pi}\int_{0}^{\infty} b_{+, -, \psi} (x, \lambda) e^{-i\Phi_{n, +, -} (x, \lambda)} d\lambda,
\end{align*}
where $\Phi_{n, +, -} (x, \lambda) =y(x,\l)+\l x_n$. We observe that
\begin{align*}
\pa_{\l}\Phi_{n, +, -} (x, \lambda)=(x+x_n)+\pa_{\l}(-\l^{-2}M_{\DD}(x)+R(x,\l))=x+x_n+O(\jap{x}^{1-\mu})
\end{align*}
uniformly in $\l\in \supp \g$ by \eqref{eq:phasedecom}, \eqref{eq:phaserembound}, and $\left| \int_{0}^{x} V(s)^k ds\right|\lesssim \jap{x}^{1-\mu}$. Therefore, it holds that
\begin{align*}
|\partial_{\lambda} \Phi_{n, +, -} (x, \lambda)| \gtrsim x+x_n+1, \quad |\partial^{\alpha} _{\lambda} \Phi_{n, +, -} (x, \lambda)| \lesssim |M_{\DD} (x)| \lesssim x
\end{align*}
for $\alpha \ge 2$, $n \in \N$ and $x \ge R_0$ if $R_0$ is sufficiently large. The integration by parts shows
\begin{align*}
|W_{+, +, -, \psi} u_n (x)|\lesssim (1+x)^{-2}
\end{align*}
uniformly in $n\in\mathbb{N}$ and $x\geq R_0$, where we have used the symbolic estimate of $b_{+,-}$ in \eqref{eq:b_+jsymbest}. Thus, the assertion follows from $(1+x)^{-2}\in L^p(\R_{\geq R_0})$.
\end{proof}

Next we give a lower bound on $W_{+, +, +, \psi} u_n$. We take small $\epsilon >0$ such that $[\lambda_1 -3\epsilon, \lambda_1 + 3\epsilon] \subset \supp \g$.

\begin{lemma}\label{264131323}
We define $\Omega_n (\epsilon):=\{x \in \R_{\ge R_0} \mid (\l_1 + \epsilon)^2 (x-x_n) < -M_{\DD} (x) <(\l_1 - \epsilon)^2 (x-x_n) \}$.
If $\epsilon >0$ is sufficiently small, we have
\begin{align*}
|W_{+, +, +, \psi} u_n (x)| \gtrsim |M_{\DD} (x_n)|^{-\frac{1}{2}},\quad \VOL (\Omega_{n} (\epsilon)) \gtrsim |M_{\DD} (x_n)|
\end{align*}
for $x \in \Omega_n(\e)$ and sufficiently large $n \in \N$.
\end{lemma}

\begin{proof}
We rewrite $W_{+, +, +, \psi} u_n$ as
\begin{align*}
W_{+, +, +, \psi} u_n (x) =\frac{1}{2\pi} \int_{0}^{\infty} b_{+, +, \psi} (x, \lambda) e^{i\Phi_{n, +, +} (x, \lambda)} d\lambda,
\end{align*}
where $\Phi_{n, +, +} (x, \lambda) =y(x, \lambda) -\lambda x_n=(x-x_n)\l-M_{\DD}(x)\l^{-1} +R (x, \lambda)$ with the notation in \eqref{eq:phasedecom}. 

Firstly, we show that $M_{\DD}(x)\sim M_{\DD}(x_n)$ for $x\in \Omega_n(\e)$ and $\VOL (\Omega_{n} (\epsilon)) \gtrsim |M_{\DD} (x_n)|$ for sufficiently large $n$.
By the mean value theorem, for $x \in \Omega_{n} (\epsilon)$, we have $\xi_{x, x_n} \in [x, x_n] \hspace{1mm} (\mathrm{or} \in [x_n, x])$ such that $M_{\DD} (x) -M_{\DD} (x_n) = \frac{1}{2} V(\xi_{x, x_n}) (x-x_n)$, which yields $|M_{\DD} (x)| \sim |M_{\DD} (x_n)|$ for $x\in \Omega_n(\epsilon)$. Moreover, there exists $\xi_{n, c} \sim x_n$ satisfying $M_{\DD} (x_n - cM_{\DD} (x_n)) -M_{\DD} (x_n) = \frac{1}{2} V(\xi_{n, c}) \cdot (-c)M_{\DD} (x_n)$, which ensures that if $|c - \l^{-2} _1| \ll 1$, we have $x_n - cM_{\DD} (x_n) \in \Omega_{n} (\epsilon)$. Therefore $\VOL (\Omega_{n} (\epsilon)) \gtrsim |M_{\DD} (x_n)|$ follows.

Put $L_{n,x}:=x-x_n$ and $\phi_{n,x}(\l)=L_{n,x}^{-1}\Phi_{n, +, +} (x, \lambda)$. We rewrite it as
\begin{align*}
\phi_{n, x} (\lambda) = \lambda - \frac{M_{\DD} (x)}{x-x_n} \frac{1}{\lambda} + r_{n, x} (\lambda), \quad r_{n, x} (\lambda) :=(x-x_n)^{-1}R(x,\l)
\end{align*}
for $x\in \Omega_n(\e)$. Then, we have have $|\pa_{\l}^{\a}r_{n,x}(\l)|=o(1)$ for $x\in \Omega_n(\e)$ and $\l\in \supp \g$ as $n\to \infty$. To justify this, we observe that by \eqref{eq:phaserembound},
\begin{align*}
&|\pa_{\l}^{\a}r_{n,x}(\l)|\lesssim |x-x_n|^{-1}\int_{0}^{x} V(s)^2 ds\sim |M_{\DD} (x_n)|^{-1}\int_{0}^{x_n} V(s)^2 ds + |M_{\DD} (x_n)|^{-1}\int_{x_n}^{x} V(s)^2 ds\\
&|M_{\DD} (x_n)|^{-1}\left|\int_{x_n}^{x} V(s)^2 ds \right| \le |M_{\DD} (x_n)|^{-1} \cdot |V (\xi_{x, x_n})|^2 |x-x_n| \lesssim   |V (\xi_{x, x_n})|^2,
\end{align*}
where we have used $|M_{\DD}(x)|\sim |M_{\DD}(x_n)|\sim |x-x_n|$. Since $\x_{x,x_n}\in [x,x_n]$, we have $\x_{x,x_n}$ with $x\in \Omega_n(\e)$ tends to $\infty$ and $|V(\x_{x,x_n})|\to 0$ as $n\to \infty$. This and Lemma \ref{2658032} prove $|\pa_{\l}^{\a}r_{n,x}(\l)|=o(1)$.

Next, we show that, if $n \in \N$ is sufficiently large and $x \in \Omega_n (\epsilon)$, there exists a unique $\lambda_{n, x} \in [\lambda_1 -2\epsilon, \lambda_1 +2\epsilon]$ such that $\partial_{\lambda} \phi_{n, x} (\lambda_{n, x}) =0$. In fact, we can see
\begin{align*}
\partial_{\lambda} \phi_{n, x} (\lambda) = 1 + \frac{M_{\DD} (x)}{x-x_n} \frac{1}{\lambda^2} +\partial_{\lambda}r_{n, x} (\lambda), \quad \partial^2 _{\lambda} \phi_{n, x} (\lambda) = -\frac{2M_{\DD} (x)}{x-x_n} \frac{1}{\lambda^3} + \partial^2 _{\lambda}r_{n, x} (\lambda).
\end{align*}
These formulae yield $\partial^2 _{\lambda} \phi_{n, x} (\lambda) >0$ on $\supp \psi$ and
\begin{align*}
&\partial_{\lambda} \phi_{n, x} (\lambda_1 -2\epsilon) \le 1-\frac{(\lambda_1 -\epsilon)^2}{(\lambda_1 -2\epsilon)^2} + o(1) <0, \\
&\partial_{\lambda} \phi_{n, x} (\lambda_1 +2\epsilon) \ge 1-\frac{(\lambda_1 +\epsilon)^2}{(\lambda_1 +2\epsilon)^2} +o(1)>0.
\end{align*}
as $n\to \infty$.
Therefore the intermediate value theorem ensures the unique existence of $\lambda_{n, x} \in [\lambda_1 -2\epsilon, \lambda_1 +2\epsilon]$. 

Now, by the quantitative stationary phase theorem \cite[Theorem 7.7.5]{Hor} with the large parameter $L_{n,x}=x-x_n$, we obtain
\begin{align*}
|W_{+, +, +, \psi} u_n (x)| \gtrsim \g(\l_{n,x}^2)L^{-\frac{1}{2}} _{n, x} \sim |M_{\DD} (x)|^{-\frac{1}{2}}\sim |M_{\DD}(x_n)|^{-\frac{1}{2}}\quad (x\in \Omega_n(\e))
\end{align*}
for sufficiently large $n$ and sufficiently small $\e$, where we have used \eqref{eq:b_+jsymbest}, $|\partial^2 _{\lambda} \phi_{n, x} (\lambda_{n, x})| \sim 1$ and $|\partial^{\alpha} _{\lambda} \phi_{n, x} (\lambda)| \lesssim 1$ for $\lambda \in \supp \psi$ and $\alpha=0,1,2$ to check the assumption. This completes the proof.
\end{proof}

\begin{proof}[Proof of Theorem \ref{264131227}]
We take $u_n$ as in \eqref{eq:u_ndef}. Since $\supp \hat{\chi}\subset (0,\infty)$, we have $W_{+, -, \sigma, \psi} u_n =0$ for $\sigma = \pm$. Lemmas \ref{264131246} implies $\|W_{+, +, -, \psi}u_n\|_{L^p (\Omega_n(\e))}\lesssim 1$.
By Lemma \ref{264131323}, we obtain 
\begin{align*}
\|W_{+, +, +, \psi}u_n\|_{L^p (\R_{\ge R_0})}^p\geq \|W_{+, +, +, \psi}u_n\|_{L^p (\Omega_n(\e))}^p\gtrsim |M_{\DD} (x_n)|^{1-\frac{p}{2}} \to \infty\quad \text{as $n\to \infty$}.
\end{align*}
Thus, the claim follows.
\end{proof}

\subsection{Unboundedness for the case $p>2$}

Let
$\psi \in C^{\infty} _c (\R; \R)$ be such that $\supp \psi \subset (2\l^2 _0, \infty)$ and not identically zero. We take $\lambda_1 \in (\sqrt{2}\l_0, \infty)$ such that $\psi (\lambda_1^2) \neq 0$. Then we have

\begin{thm}\label{264221222}
Let $p \in [2, \infty]$ and suppose $M_{\DD} \notin L^{\infty}(\mathbb{R})$. The middle energy wave operator $W_+ \psi (P_0)$ is bounded on $L^p (\R)$ if and only if $p=2$.
\end{thm}

We may assume $M_{\DD} \notin L^{\infty}(\mathbb{R}_{>0})$. Let $x_n$ be a sequence in Lemma \ref{2658032} again. We set 
\begin{align*}
\Psi_{+, +} (x, x', \lambda) = y(x, \l) -\l x'=\l(x-x')-\l^{-1}M_{\DD}(x)+R(x,\l)
\end{align*}
with the notation in \eqref{eq:phasedecom} so that $W_{+, +, +, \psi} (x, x') = \int_{0}^{\infty} b_{+, +, \psi} (x, \lambda) e^{i\Psi_{+, +} (x, x', \lambda)}d\lambda$. For $0<C_1<C_2$, we define
\begin{align}\label{eq:Omegadef}
&\Omega:=\{(x,x')\in \R^2\mid -C_2(x-x')< M_{\mathrm{D}}(x)< -C_1(x-x') \}.
\end{align}
As a first step, we see that the term $R(x,\l)$ can be regarded as a perturbation of $M_{\DD}(x)$ if $x$ is close to the sequence $x_n$.

\begin{lemma}\label{lem:RperturbMd}
For $\a\in\mathbb{N}_0$ and sufficiently large $n\in\mathbb{N}$, we have $|M_{\DD}(x)|\gg 1$ and $|\pa_{\l}^{\a}R(x,\l)|\ll |M_{\DD}(x)|$ if $|x-x_n|<1$ and $\l\in \supp \g$. Moreover, we have $M_{\DD}(x)\sim M_{\DD}(x_n)$ when $|x-x_n|<1$ with sufficiently large $n$.
\end{lemma}

\begin{proof}
We employ \eqref{eq:phaserembound} to obtain
\begin{align*}
|\pa_{\l}^{\a}R(x_n,\l)|\lesssim \int_0^{x_n}|V(s)|^2ds,\quad |\pa_{\l}^{\a}(R(x,\l)-R(x_n,\l))|\lesssim \left|\int_{x_n}^xV(s)^2ds\right|\lesssim |x-x_n|\leq 1
\end{align*}
for $|x-x_n|\leq 1$ and $\l\in \supp \g$. Moreover, $|M_{\DD}(x)|\geq |M_{\DD}(x_n)|-\frac{1}{2}\left|\int_{x_n}^xV(s)ds\right| \gtrsim |M_{\DD}(x_n)|\gg 1$ for $|x-x_n|\leq 1$ and sufficiently large $n$. Thus, the first assertion follows from the choice of the sequence $x_n$ in Lemma \ref{2658032}. The second assertion follows from $|M_{\DD}(x)-M_{\DD}(x_n)|=\frac{1}{2}|\int_{x_n}^xV(s)ds|\leq \frac{1}{2}\|V\|_{L^{\infty}}|x-x_n|\lesssim 1 \ll |M_{\DD}(x_n)|$ when $|x-x_n|<1$ with sufficiently large $n$. This implies  $M_{\DD}(x)\sim M_{\DD}(x_n)$ in that case.
\end{proof}

\begin{proposition}\label{prop:p>2stphase}
For $R_2\gg R_1\gg 1$, $C_2\gg 1$, $0<C_1\ll 1$ and sufficiently large $n\in\mathbb{N}$, we have the following:

\vspace{1mm}
\noindent$(i)$ We have
\begin{align*}
|W_{+,+,+,\g}(x,x')|\lesssim (1+|x-x'|)^{-2}
\end{align*}
if $|x-x_n|<1$ and $(x,x')\in \Omega^c\cup \{|x| \le R_1, |x'| \ge R_2 \} \cup \{|x| \ge R_2, |x'| \le R_1\}\cup \{|x|,|x'|\leq R_2\}$.

\vspace{1mm}
\noindent$(ii)$ Set $\phi_{x, x'} (\lambda):=(x-x')^{-1}\Psi_{+,+}(x,x',\l)$. If $|x-x_n|<1$ and $(x,x')\in \Omega_1:=\Omega\cap \{x,x'\geq R_1 \}$, then there exists a unique non-degenerate critical point $\l_{x,x'}\in (C_1,C_2)$ of $\phi_{x, x'} (\lambda)$  such that
\begin{align*}
&W_{+,+,+,\g}(x,x')=|x-x'|^{-\frac{1}{2}} b_{+, +,\g} (x, \l_{x, x'}) e^{i(x-x')\phi_{x, x'} (\l_{x, x'}) -\frac{i\pi}{4}}+ O(|x-x'|^{-1}),\\
&\VOL (\{ x' \in \R \mid  (x,x'),  (x_n, x')\in \Omega_1,\,\, |\l_{x,x'}-\l_1|<\e'\})\sim |M_{\DD}(x_n)|\quad \text{for $|x-x_n|<1$}
\end{align*}
for sufficiently small $\e'>0$, where we recall that $\l_1$ satisfies $\g(\l_1^2)\neq 0$.

\end{proposition}

\begin{proof}

\noindent$(i)$ The estimate for $|x|,|x'|\leq R_2$ is obvious due to the compactness of $\supp \g$. Thus, we consider the other cases.

We show that for $R_2\gg R_1\gg 1$ and $0<C_1\ll 1\ll C_2$, we have
\begin{align}\label{eq:Psinonstest}
|\pa_{\l} \Psi_{ +, +} (x, x', \lambda)| \gtrsim |x-x'| + |M_{\DD} (x)|, \quad |\pa^{\alpha} _{\l} \Psi_{ +, +} (x, x', \lambda)| \lesssim |M_{\DD} (x)|
\end{align}
for any $\alpha \ge 2$ if $|x-x_n|<1$, $(x, x') \in \Omega^{c} \cup \{|x| \le R_1, |x'| \ge R_2 \} \cup \{|x| \ge R_2, |x'| \le R_1\}$ and $\l\in \supp \g$ with $n$ sufficiently large. To do this, we observe that $|M_{\DD}(x)|\lesssim \max(\jap{x}^{1-\m},\log\jap{x})$, which follows from \eqref{26641742} and the definition of $M_{\DD}(x)=\frac{1}{2}\int_0^xV(s)ds$. Moreover, we have $|M_{\DD}(x)|\ll |x-x'|$ for $(x,x')\in \Omega^c$ if $0<C_1\ll 1\ll C_2$ and $|x-x'|\gtrsim \max(|x|,|x'|)$ for $\{|x| \le R_1, |x'| \ge R_2 \} \cup \{|x| \ge R_2, |x'| \le R_1\}$ if $R_2\gg R_1\gg 1$. Thus
\begin{align*}
|\pa_{\l} \Psi_{ +, +} (x, x', \lambda)| \gtrsim |x-x'| + |M_{\DD} (x)|-|\pa_{\l}R(x,\l)|, \quad |\pa^{\alpha} _{\l} \Psi_{ +, +} (x, x', \lambda)| \lesssim |M_{\DD} (x)|+|\pa_{\l}^{\a}R(x,\l)|
\end{align*}
hold for such $\a$, $(x,x')$, and $\l$. On the other hand, we have $|\pa_{\l}^{\a}R(x,\l)|\ll |M_{\DD}(x)|$ if $|x-x_n|\leq 1$ and $\l\in \supp\g$ with sufficiently large $n$ by Lemma \ref{lem:RperturbMd}. Thus \eqref{eq:Psinonstest} follows. Now, the part $(i)$ is proved by \eqref{eq:b_+jsymbest} integrating by parts twice in the expression \eqref{eq:waveopdecom}. 

\vspace{1mm}
\noindent$(ii)$ We write the integral kernel as $W_{+, +, +, \psi}(x,x')=\int_{0}^{\infty} b_{+, +, \psi} (x, \lambda) e^{i(x-x')\phi_{x, x'} (\lambda)}d\lambda$. We see that $\phi_{x, x'} (\lambda) = \l - \frac{1}{\l} \frac{M_{\DD}(x)}{x-x'} + (x-x')^{-1}R(x,\l)$ holds in the notation of \eqref{eq:phasedecom}. Then, it is easy to check that if $R_1 \gg 1$, $0<C_1 \ll 1\ll C_2$, and $n\in\mathbb{N}$ is sufficiently large, we have
\begin{align*}
\pa^2 _{\l} \phi_{x, x'} (\lambda) >0, \quad \pa_{\l} \phi_{x, x'} (C_1) <0, \quad \pa_{\l} \phi_{x, x'} (C_2) >0, \quad |\pa^{\alpha} _{\l} \phi_{x, x'} (\lambda)| \lesssim 1
\end{align*}
uniformly in $(x, x') \in \Omega_1$ with $|x-x_n|<1$ and $\lambda \in \supp \psi$ by Lemma \ref{lem:RperturbMd}. Therefore the intermediate value theorem yields the unique existence of $\lambda_{x, x'} \in (C_1, C_2)$ such that $\pa_{\l} \phi_{x, x'} (\l_{x, x'}) =0$. By the stationary phase theorem \cite[Theorem 7.7.5]{Hor} with the large parameter $|x-x'|=-(x-x')$, we have
\begin{align*}
W _{+, +, +, \psi} (x, x')= (2\pi)^{-1}|x-x'|^{-\frac{1}{2}} b_{+, +,\g} (x, \l_{x, x'}) e^{i(x-x')\phi_{x, x'} (\l_{x, x'}) -\frac{i\pi}{4}}+ O(|x-x'|^{-1}),
\end{align*}
for $|x-x_n|<1$ and $(x,x')\in \Omega_1$, where the factor $-$ in front of $\frac{i\pi}{4}$ comes from $x-x'<0$.

To obtain the lower bound of the volume, we observe that the critical point $\l_{x,x'}$ satisfies
\begin{align*}
0=(\pa_{\l}\phi_{x,x'})(x,\l_{x,x'})=1+\l_{x,x'}^{-2}(x-x')M_{\DD}(x)+(x-x')^{-1}\pa_{\l}R(x,\l_{x,x'}).
\end{align*}
By Lemma \ref{lem:RperturbMd}, we have $\l_{x,x'}\sim \sqrt{-\frac{M_{\DD}(x)}{x-x'}}$ for $|x-x_n|<1$ with sufficiently large $n$. Using this and $|M_{\DD}(x)|\sim M_{\DD}(x_n)$, we have
\begin{align*}
&\VOL (\{ x' \in \R \mid  (x,x'), (x_n,x')\in \Omega_1,\,\, |\l_{x,x'}-\l_1|<\e'\})\\
&\sim \VOL (\{ x' \in \R \mid  (x,x')\in \Omega_1\})\sim |M_{\DD}(x)|\sim |M_{\DD}(x_n)|
\end{align*}
if $R_1$ is taken sufficiently large.
\end{proof}

Before going to the proof of Theorem \ref{264221222}, we need one more notation. For $S\subset \R^2$ and a linear operator $T$, we write its integral kernel by $T(x,x')$ and
\begin{align}\label{eq:intrest}
T^S(x,x'):=1_{S}(x,x')T(x,x')
\end{align}
and denote the corresponding operator by $T^{S}$, where we recall $1_S$ denotes the characteristic function.

\begin{proof}[Proof of Theorem \ref{264221222}]

If $W_{+}\g(P_0)$ was bounded on $L^p$, then $W_{+}\g(P_0)\f(D_x)$ must be bounded on $L^p$ for $\f\in C_c^{\infty}(\R)$ since $\f(D_x)\in B(L^p)$. Thus, it suffices to prove $W_{+}\g(P_0)\f(D_x)$ is not bounded on $L^p$ for some $\f$. Taking $\supp \f\subset (0,\infty)$ and $\f=1$ on $\supp \g(|\x|^2)$, we have $W_{+,+,\s_2,\g}\f(D_x)=W_{+,+,\s_2,\g}$ and $W_{+,-,\s_2,\g}\f(D_x)=0$. Thus, it suffices to construct a sequence $u_n\in L^p(\R)$ such that 
\begin{align*}
\|u_n\|_{L^p}\sim 1,\quad \sup_{n}\|W_{+,+,-,\g}u_n\|_{L^p}<\infty,\quad \lim_{n\to\infty}\|W_{+,+,+,\g}u_n\|_{L^p}=\infty.
\end{align*}

We take $u_n (x'):= |M_{\DD} (x_n)|^{-\frac{1}{p}} 1_{\Omega_1} (x_n, x') e^{-i(x_n-x')\phi_{x_n, x'} (\l_{x_n, x'})+\frac{i\pi}{4}}$, where $\Omega_1$ is defined in Proposition \ref{prop:p>2stphase} $(ii)$. By the definition of $\Omega$ in \eqref{eq:Omegadef} and $\Omega_1$, we have $\|u_n\|_{L^p}\sim 1$ uniformly in $n$. Moreover, similarly to the proof of Lemma \ref{264131246}, the non-stationary phase argument shows that $\sup_{n}\|W_{+,+,-,\g}u_n\|_{L^p}<\infty$ holds. It remains to prove $\|W_{+,+,+,\g}u_n\|_{L^p(|x-x_n|<\e)}\to \infty$ as $n\to \infty$ for sufficiently small $\e>0$. 

Set $\Omega_2:=\{(x,x')\in \Omega\mid x< -R_1,\,\, x'< -R_1 \}$ with $R_1$ appearing in Proposition \ref{prop:p>2stphase}. By the definition of $\Omega_1$, we have $\supp u_n\subset \R_{>0}$ and hence $W_{+,+,+,\g}^{\Omega_2}u_n(x)=0$. Since the subset $(\Omega_1\cup \Omega_2)^c$ is contained in the set appearing in Proposition \ref{prop:p>2stphase} $(i)$, Young's inequality shows that $\sup_n\|W_{+,+,+,\g}^{(\Omega_1\cup \Omega_2)^c}u_n\|_{L^p(|x-x_n|<\e)}<\infty$. Combining them with the decomposition $W_{+,+,+,\g}=W_{+,+,+,\g}^{\Omega_1}+W_{+,+,+,\g}^{\Omega_2}+W_{+,+,+,\g}^{(\Omega_1\cup \Omega_2)^c}$, we only need to prove $\|W_{+,+,+,\g}^{\Omega_1}u_n\|_{L^p(|x-x_n|<\e)}\to \infty$ as $n\to \infty$.

To accomplish this, we write by Proposition \ref{prop:p>2stphase} $(ii)$
\begin{align*}
W^{\Omega_1} _{+, +, +, \psi} u_n (x) =(2\pi)^{-1} |M_{\DD} (x_n)|^{-\frac{1}{p}} \int_{A_{x, n}} (x-x')^{-\frac{1}{2}} b_{+, +,\g} (x, \l_{x, x'}) e^{i\mathcal{P} (n, x, x')} dx'+O(|M_{\DD} (x_n)|^{-\frac{1}{p}})
\end{align*}
if $|x-x_n|<1$, where 
\begin{align*}
\mathcal{P} (n, x, x'):=(x-x')\phi_{x, x'} (\l_{x, x'}) - (x_n-x')\phi_{x_n, x'} (\l_{x_n, x'}),\,\, A_{x, n} := \{x' \in \R \mid (x, x'), (x_n, x') \in \Omega_1\},
\end{align*}
and we have used $\VOL(A_{x,n})\lesssim |M_{\DD}(x_n)|$.

We show that $|\mathcal{P} (n, x, x')| \lesssim \epsilon$ holds when $x' \in A_{x,n}$ and $|x-x_n| < \epsilon$ hold. To do this, we recall $\phi_{x, x'} (\lambda) = \l - \frac{1}{\l} \frac{M_{\DD}(x)}{x-x'} + (x-x')^{-1}R(x,\l)$ from Proposition \ref{prop:p>2stphase} and \eqref{eq:phasedecom} and perform a calculation as
\begin{align*}
\mathcal{P} (n, x, x')
&= (x-x_n) \phi_{x, x'} (\l_{x, x'}) + (x_n-x') (\phi_{x, x'} (\l_{x, x'}) -\phi_{x_n, x'} (\l_{x, x'})) \\
&\quad \quad + (x_n-x') (\phi_{x_n, x'} (\l_{x, x'}) -\phi_{x_n, x'} (\l_{x_n, x'})).
\end{align*}
The first term can be estimated as $|(x-x_n) \phi_{x, x'} (\l_{x, x'})| \lesssim \epsilon$ if $(x, x'), (x_n, x') \in \Omega_1$ and $|x-x_n| < \epsilon$. Moreover we have $|\pa_x \phi_{x, x'} (\lambda)| \lesssim |x-x'|^{-1}$ for $\l \in \supp \psi$ and $(x, x') \in \Omega_1$, which yields
\begin{align*}
|(x_n-x') (\phi_{x, x'} (\l_{x, x'}) -\phi_{x_n, x'} (\l_{x, x'}))| \lesssim |x_n-x'| \left|\int_{x_n}^{x} |s-x'|^{-1} ds\right| \lesssim \epsilon
\end{align*}
as long as $(x_n, x') \in \Omega_1$ and $|x-x_n| < \epsilon$. Concerning the third term, using the relation
\begin{align*}
(\pa_{\l}\pa_x \phi_{x, x'}) (\l_{x, x'}) + (\pa^2 _{\l} \phi_{x, x'}) (\l_{x, x'}) \pa_{x} \l_{x, x'} =0,
\end{align*}
we obtain $|\pa_{x} \l_{x, x'}| \le |(\pa_{\l}\pa_x \phi_{x, x'}) (\l_{x, x'})| \cdot |(\pa^2 _{\l} \phi_{x, x'}) (\l_{x, x'})|^{-1} \lesssim |x-x'|^{-1}$ for $(x, x') \in \Omega_1$, which yields
\begin{align*}
|(x_n-x') (\phi_{x_n, x'} (\l_{x, x'}) -\phi_{x_n, x'} (\l_{n, x'}))| \lesssim |x'-x_n| |\l_{x, x'} -\l_{x_n, x'}| \lesssim  |x'-x_n| \left| \int_{n}^{x} |s-x'|^{-1} ds\right| \lesssim \epsilon
\end{align*}
if $(x_n, x') \in \Omega_1$ and $|x-x_n| < \epsilon$. In this way, $|\mathcal{P} (n, x, x')| \lesssim \epsilon$ is proved.

Now, we take $0<\e\ll1$ and $n\gg 1$ so that $\mathrm{Re}\, (e^{i\mathcal{P}(n,x,x')}b_{+,+,\g}(x,\l_{x,x'}))>\frac{1}{2}\g(\l_{x,x'}^2)$ for $x' \in A_{x,n}$ and $|x-x_n| < \epsilon$, which is possible by \eqref{eq:b++--asymp} and $|\mathcal{P} (n, x, x')| \lesssim \epsilon$. Taking $\e'>0$ sufficiently small so that $\g(\l^2)\gtrsim 1$ for $|\l-\l_1|<\e'$, we obtain
\begin{align*}
&|W^{\Omega_1} _{+, +, +, \psi} u_n (x)|\geq \mathrm{Re}\, W^{\Omega_1} _{+, +, +, \psi} u_n (x)\\
&\gtrsim \frac{1}{2}|M_{\DD} (x_n)|^{-\frac{1}{2}-\frac{1}{p}}\VOL(A_{x,n}\cap \{x'\in \R\mid |\l_{x,x'}-\l_1|<\e'\})+O(|M_{\DD} (n)|^{-\frac{1}{p}})\\
&\gtrsim \frac{1}{2}|M_{\DD} (x_n)|^{\frac{1}{2}-\frac{1}{p}}
\end{align*}
when $|x-x_n|<\e$ and $n\gg 1$, where we have used the volume bound in Proposition \ref{prop:p>2stphase} $(ii)$. Consequently, we obtain $\|W^{\Omega_1} _{+, +, +, \psi} u_n\|_{L^p(|x-x_n|)}\gtrsim |M_{\DD} (x_n)|^{\frac{1}{2}-\frac{1}{p}}\to \infty$ as $n\to \infty$.
This completes the proof.
\end{proof}

\section{The modified wave operators in the low energy regime}\label{2508241543}

Throughout of this section, we assume \eqref{26641742} and
\begin{align}\label{eq:Vneg}
0<\mu<2,\quad  -V(x)\gtrsim \jap{x}^{-\mu}.
\end{align}

\subsection{Isozaki--Kitada modifiers and modified wave operators in the low energy regime}\label{subsection:lowIsoKita}
Here, we introduce modified wave operators adapted to the study in the low energy regime. The construction here is closely related to the one by Derezi\'nski--Skibsted \cite{DS}.
Let $y(x,\l)$ and $\phi_{\pm}(x,\x)$ as in  \eqref{eq:yphidef} which are well-defined for $x\in \R$ and $\x\in \R\setminus \{0\}$ due to the assumption \eqref{eq:Vneg}\footnote{In general, $\phi_{\pm}$ is discontinuous at $\x=0$.}. Let $\overline{\chi}\in C^{\infty}(\R;[0,1])$ such that $\overline{\chi}(x)=1$ for $x\geq 2$ and $\overline{\chi}(x)=0$ for $x\leq 1$. We define $a_{\pm}(x,\x):=\overline{\chi}(\pm x/R_0)1_{\R_{> 0}}(\x)+\overline{\chi}(\mp x/R_0)1_{\R_{< 0}}(\x)$ for a fixed $R_0\gg 1$ and the Isozaki--Kitada modifier $J_{\pm}$ by the expression \eqref{eq:Jpmdef}, which can be understood as linear continuous operators from $\mathcal{S}(\R)$ to $\mathcal{S}'(\R)$. We note that $J_{\pm}f\in L^2(\R)$ for $\hat{f}\in C_c^{\infty}(\R\setminus \{0\})$. 
Similarly to \cite[Theorem 6.1]{DS}, the limits
\begin{align*}
W_{\pm}f:=\lim_{t\to \pm\infty}e^{itP} J_{\pm} e^{-itP_0}f\quad (\hat{f}\in C_c^{\infty}(\R\setminus \{0\}))
\end{align*}
exist and define\footnote{We note that the discontinuity of $\phi_{\pm}$ and $a_{\pm}$ is nothing to do with the existence and the asymptotic completeness of $W_{\pm}$ since $\x=0$ corresponds to the zero energy of $P_0$, whose Lebesgue measure is zero.} the isometries in $L^2(\R)$ and the ranges of $W_{\pm}$ are the absolutely continuous subspace of $P$. The operators $W_{\pm}$ are called \textbf{modified wave operators}. 

\begin{remark}
We note that the modified wave operators $W_{\pm}$ introduced here are not the same as \eqref{eq:IKmodwave} on $L^2(\R)$ but coincide with them on $\mathrm{Ran}~E_{P_0}([\sqrt{2}\l_0,\infty))$. Therefore, we continue to use the same notation.
\end{remark}

\subsection{Relationship to the modified wave operators by Derezi\'nski--Skibsted}\label{eq:sevwavops2}

Let us compare our modified operators and the ones constructed by Derezi\'nski--Skibsted \cite{DS} when $-V(x)\gtrsim \jap{x}^{-\mu}$ and $V$ is spherically symmetric, that is, $V$ is an even function in the one-dimensional case.

\begin{lemma}
Suppose that $V$ is spherically symmetric, that is $V$ is an even function. For each $1\leq p\leq \infty$, the $L^p$-boundedness of $W_{\pm}$ and that of the Derezi\'nski--Skibsted modified wave operators $W_{\pm,\mathrm{DS}}$ is equivalent.

\end{lemma}

\begin{proof}
We only deal with $W_+$. As in the formula \cite[(3.9)]{DS}, they use the phase function\footnote{The expression here is slightly different from the one in \cite[(3.9)]{DS} since their Hamiltonian is $-\frac{1}{2}\pa_x^2+V(x)$ instead of $-\pa_x^2+V(x)$.} as $\phi_{+,\mathrm{DS}}(x,\x)=|\x|R_0+\int_{R_0}^{|x|}\sqrt{\x^2-V(s)}ds $ for $|x|\geq R_0$ with $\sgn(x\x)=1$, where $R_0\geq 1$ is a fixed constant. Moreover, $\phi_{+}(x,\x)=\int_0^{|x|}\sqrt{\x^2-V(s)}ds$ for $\sgn(x\x)=1$ due to the evenness of $V$. Therefore,
\begin{align*}
\phi_+(x,\x)-\phi_{+,\mathrm{DS}}(x,\x)=-|\x|R_0+\int_0^{R_0}\sqrt{\x^2-V(s)}ds=\int_0^{R_0}\frac{-V(s)}{|\x|+\sqrt{\x^2-V(s)}}=:A(\x)
\end{align*}
when $\sgn(x\x)=1$ and the relationship between $W_+$ and their modified wave operator $W_{+,\mathrm{DS}}$ is given by
\begin{align*}
W_+=W_{+,\mathrm{DS}}e^{iA(D_x)}.
\end{align*}
Thus, it suffices to prove $e^{iA(D_x)}$ is bounded on $L^p$.

We observe 
\begin{align}\label{eq:DScompsymbol}
|\pa_{\x}^{\a}e^{iA(\x)}|\lesssim |\x|^{-\a}\,\, \text{for  $0<|\x|\leq 1$},\quad |\pa_{\x}^{\a}(e^{iA(\x)}-1)|\lesssim \jap{\x}^{-1-\a}\,\, \text{for $|\xi |\geq 1$}.
\end{align}
Let $\g\in C_c^{\infty}(\R)$ such that $\g(\x)=1$ for $|\x|\leq 1/2$ and $\g(\x)=0$ for $|\x|\geq 1$. We write $e^{iA(D_x)}(I-\g(D_x))=(e^{iA(D_x)}-I)(I-\g(D_x))+I-\g(D_x)$. By the second estimate in \eqref{eq:DScompsymbol} and \cite[\S7, Lemma 5.2]{Tay}, $(e^{iA(D_x)}-I)(I-\g(D_x))$ and $\g(D_x)$ are bounded on $L^p$. Since the identity operator $I$ is bounded on $L^p$, the $L^p$-boundedness of $e^{iA(D_x)}(I-\g(D_x))$ follows. It remains to show that $e^{iA(D_x)}\g(D_x)$ is bounded on $L^p$.

By Young's inequality, we have only to show $\mathcal{F} [e^{iA(\xi)} \psi (\xi)] \in L^1 (\R)$. By the support property of $\psi$, we know $\mathcal{F} [e^{iA(\xi)} \psi (\xi)] \in L^{\infty} (\R)$. Next, integration by parts yields
\begin{align*}
\mathcal{F} [e^{iA(\xi)} \psi (\xi)] (x) = \frac{1}{ix} \int_{\R} ie^{iA(\xi)} \psi (\xi) A^{'} (\xi) e^{-ix\xi} d\xi + \frac{1}{ix} \int_{\R} e^{iA(\xi)} \psi^{'} (\xi) e^{-ix\xi}d\xi =: \mathrm{I}(x)+\mathrm{II}(x).
\end{align*}
Since $A^{'} (\xi)$ is bounded for $\xi \in \R \setminus \{0\}$, we have $|\mathrm{II} (x)| \lesssim |x|^{-2}$ by integrating by parts once again. Concerning $\mathrm{I}(x)$, we calculate, with $\epsilon \in (0, \frac{1}{10})$,
\begin{align*}
|x|^{1+\epsilon} |\mathrm{I} (x)| = |x|^{\epsilon} |\mathcal{F} [e^{iA(\xi)} \psi (\xi) A^{'} (\xi)] (x)| \lesssim \||D_{\xi}|^{\epsilon} (e^{iA(\xi)} \psi (\xi) A^{'} (\xi))\|_{L^1}.
\end{align*}
By the Kato--Ponce inequality \cite[Theorem 1.1]{G3}, we obtain
\begin{align*}
\||D_{\xi}|^{\epsilon} (e^{iA(\xi)} \psi (\xi) A^{'} (\xi))\|_{L^1} &\lesssim \||D_{\xi}|^{\epsilon} (e^{iA(\xi)} \psi (\xi))\|_{L^2} \|\tilde{\psi}A^{'}\|_{L^2} + \|e^{iA(\xi)} \psi (\xi)\|_{L^2} \||D_{\xi}|^{\epsilon} (\tilde{\psi}A^{'})\|_{L^2} \\
&\lesssim \||D_{\xi}|^{\epsilon} (e^{iA(\xi)} \psi (\xi))\|_{L^2} + \||D_{\xi}|^{\epsilon} (\tilde{\psi}A^{'})\|_{L^2},
\end{align*}
where $\tilde{\psi} \in C^{\infty} _c (\R)$ satisfies $\tilde{\psi} (\xi) =1$ on the support of $\psi$. Using $\|e^{iA} \psi\|_{L^2} + \||D_{\xi}| (e^{iA} \psi)\|_{L^2} < \infty$ and an interpolation argument, we know the first term on the RHS is finite. Concerning the second term, we note $A'(\x)=-R_0\x/|\x|+\int_0^{R_0}\x/\sqrt{\x^2-V(s)}ds$. Since $\int_{0}^{R_0} (\xi/\sqrt{\xi^2 -V(s)}) ds$ is smooth, we have $|D_{\xi}|^{\epsilon} (\tilde{\psi} (\xi) \int_{0}^{R_0} (\xi/\sqrt{\xi^2 -V(s)}) ds) \in L^2$.  Moreover, $|D_{\xi}|^{\epsilon} (\tilde{\psi} (\xi) (\xi/|\xi|)) \in L^2$ due to the pointwise estimate $||D_{\xi}|^{\epsilon} (\tilde{\psi} (\xi) (\xi/|\xi|)) (x)| \lesssim |x|^{-\epsilon} \jap{x}^{-1}$, which can be checked by an elementary calculation using the singular integral expression of fractional derivatives \cite[(3.1)]{Di} (and also by \cite[Proposition 3.3]{Di}). Hence $|\mathrm{I} (x)| \lesssim |x|^{-1-\epsilon}$ and we obtain $\mathcal{F} [e^{iA(\xi)} \psi (\xi)] \in L^1 (\R)$.   
\end{proof}

%\begin{remark}
%It is plausible that the operator $e^{iA(D_x)}\g(D_x)$ is also bounded on $L^1$ and $L^{\infty}$ in our case; however, we do not pursue it here.
%\end{remark}

\subsection{Oscillatory integral representation in the low energy regime}

Similarly to Subsection \ref{subsection:Jost}, let $u_{\pm}(x,\l)$ be the Jost solutions to the equation \eqref{eq:stationarySch}, which satisfy the asymptotic expansion of the form \eqref{eq:upasymp}. The main difference from that in Subsection \ref{subsection:Jost} is that $u_{\pm}(x,\l)$ can be defined for all $x\in \R$ and $\l\geq 0$ due to the additional assumption \eqref{eq:Vneg}. We also recall that $\mathrm{Wr}(\lambda)= u_{+} (x, \lambda) \partial_{x} u_{-} (x, \lambda) - \partial_{x} u_{+} (x, \lambda) u_{-} (x, \lambda)$ is the Wronskian which is also well-defined for $\l\geq 0$. Since $u_{\pm}(x,\l)$ is continuous for $\l\geq 0$, $\mathrm{Wr}(\lambda)$ is also continuous for $\l\geq 0$.

As in Proposition \ref{prop:waveexpression1} in the middle and high energy regimes, the integral kernel of the modifier wave operator $W_{+}$ has the following expression in the low energy regime:

\begin{proposition}\label{prop:waveexpression2}
Let $\g\in C_c^{\infty}(\R)$. Then we have
\begin{align}\label{eq:waveopdecom2}
W_{+}\g(P_0)=\sum_{\s_1,\s_2\in\{\pm\}}W_{+,\s_1,\s_2},\quad W_{+,\s_1,\s_2}(x,x')=\frac{1}{2\pi}\int_0^{\infty} b_{\s_1,\s_2}(x,\l)\g(\l^2)e^{i(\s_2y(x,\l) -\s_1\l x')}d\l,
\end{align}
where $b_{\s_1,\s_2}$ satisfies $|\pa_{\l}^{\a}b_{\s_1,\s_2}(x,\l)|\lesssim \l^{-\a}$ uniformly in $x\in \R$ and $\l>0$ for $\a=0,1,2$. Moreover,
\begin{align}\label{eq:b++asymp2}
b_{+,+}(x,\l)= 2i\l^{\frac{1}{2}}(-V(x))^{-\frac{1}{4}}(\overline{\mathrm{Wr}(0)})^{-1}\left(1+o(1)\right)\quad \text{as $\l\to 0$, $x\to -\infty$ with $\l=o(\jap{x}^{-\mu/2})$ }
\end{align}
and
\begin{align}\label{eq:lowb+-imp}
|b_{+,-}(x,\l)|\lesssim \l^{\frac{1}{2}}\jap{x}^{\frac{\m}{4}+\m-2}
\end{align}
uniformly in $x<0$ and $\l>0$. 

\end{proposition}

\begin{remark}

Similarly, we also have $|b_{-,+}(x,\l)|\lesssim \l^{\frac{1}{2}}\jap{x}^{\frac{\m}{4}+\m-2}$ for $x>0$ and
\begin{align*}
b_{-,-}(x,\l)= 2i\l^{\frac{1}{2}}(-V(x))^{-\frac{1}{4}}(\overline{\mathrm{Wr}(0)})^{-1}(1+o(1))\quad \text{as $\l\to 0$, $x\to \infty$ with $\l=o(\jap{x}^{-\mu/2})$ }.
\end{align*}

\end{remark}

\begin{proof}
The proof is almost same as in Proposition \ref{prop:waveexpression1} by setting
\begin{align*}
&b_{\s_1,\s_2}(x,\l)=\frac{2i\l^{\frac{1}{2}}}{\overline{\mathrm{Wr}(\l)}(\l^2-V(x))^{\frac{1}{4}}}\overline{\tilde{a}_{\sigma_1',\sigma_2'}(x,\l)},
\end{align*}
where $\s_j'=-\s_j$. The estimate \eqref{eq:lowb+-imp} follows from $|\tilde{a}_{-,+}(x,\l)|\lesssim \jap{x}^{\m-2}$ for $x<0$, which is proved in \cite[(3.31)]{HT}. By \cite[Theorem 3.11 (i)]{HT}, we have $\tilde{a}_{-,-}(x,\l)=1+o(1)$ as $x\to -\infty$ uniformly in $\l\geq 0$. Moreover, $\l^{\frac{1}{2}}(\l^2-V(x))^{-\frac{1}{4}}=\l^{\frac{1}{2}}(-V(x))^{-\frac{1}{4}}+o(1)$ as $\l\to 0$, $x\to -\infty$ with $\l=o(\jap{x}^{-\mu/2})$. Combining them and the continuity of $\mathrm{Wr}(\l)$ at $\l=0$, we obtain \eqref{eq:b++asymp2}.
\end{proof}

\subsection{Unboundedness for $p<2$ in the low energy regime} In the following, we focus on $W_+$ only and the other case is similarly handled.
We fix $c>0$. Let $\g_c \in C^{\infty} _c ([0, \infty); [0, 1])$ satisfy $\g_c (\lambda)=1$ for $\lambda \in [0, c]$ and $\supp \g_c \subset [0, 2c]$.

\begin{thm}\label{264112222}
Let $p \in [1, 2]$. The low energy wave operator $W_+ \g_c (P_0)$ is bounded on $L^p (\R)$ if and only if $p=2$.
\end{thm}

\begin{proof}[Proof of Theorem \ref{264112222}]

Let $\chi \in C^{\infty} _c (\R; [0, 1])$ be such that $\supp \chi \subset [\frac{1}{2}, 2]$ and $\chi (\lambda) =1$ on $[\frac{2}{3}, \frac{5}{3}]$. We set $c_n := \frac{3}{2n}$ and $v_n (\lambda) = \chi(n(\lambda -c_n)) \in C^{\infty} _c (\R)$. Note that, since we have $\mathcal{F}^* [v_n] (x) = \frac{1}{n} \mathcal{F}^* [\chi] (\frac{x}{n})e^{ic_nx}$ and $\|\mathcal{F}^* [v_n]\|_{L^p (\R)} \sim n^{\frac{1}{p} -1}$, the function $u_n (x):= n^{1-\frac{1}{p}} \mathcal{F}^* [v_n] (x)$ satisfies $\|u_n\|_{L^p (\R)} \sim 1$ uniformly in $n \in \N$. 

We observe that $W_{+, -, \pm, \g_c}u_n=0$ holds by the support property of $\chi$. Therefore, we have $W_+ \g_c (P_0)u_n=W_{+, +, +, \g_c}u_n+W_{+, +, -, \g_c}u_n$. Choosing $0<\d\leq 1$ determined later, we define 
\begin{align*}
\Omega_n:=\{x<0\mid \d n\leq |x|^{\frac{1}{2}+\frac{\mu}{4}}\leq 2\d n\}.
\end{align*}
We will prove
\begin{align*}
\|W_{+,+,+,\g_c}u_n\|_{L^p(\Omega_n)}\gtrsim n^{\frac{2-\m}{2+\m}\left(\frac{1}{p}-\frac{1}{2}\right)},\quad \|W_{+, +, -, \g_c} u_n\|_{L^p (\Omega_n)} =o(n^{\frac{2-\m}{2+\m}\left(\frac{1}{p}-\frac{1}{2}\right)})\quad \text{as $n\to \infty$}.
\end{align*}
Assuming them and $1\leq p<2$, we obtain $\|W_{+, \g_c} u_n\|_{L^p (\Omega_n)}\gtrsim n^{\frac{2-\m}{2+\m}\left(\frac{1}{p}-\frac{1}{2}\right)} \to \infty$ as $n\to \infty$ though $u_n$ is bounded in $L^p (\R)$. Thus, it suffices to prove the above two estimates.

\underline{\textbf{Estimates on $W_{+, +, +, \g_c}$}} We write 
\begin{align*}
y(x, \lambda) - \int_{0}^{x} \sqrt{-V(s)} ds = \lambda^2 \int_{0}^{x} \frac{1}{\sqrt{\lambda^2 -V(s)} + \sqrt{-V(s)}} ds=:\Phi(x,\l).
\end{align*}
For sufficiently large $n \in \N$ such that $\g_c(\l) \chi(n(\lambda -b_n))= \chi(n(\lambda -b_n))$, we have
\begin{align*}
W_{+, +, +, \g_c} u_n (x) =  n^{1-\frac{1}{p}}e^{i\int_{0}^{x} \sqrt{-V(s)} ds} \int_{0}^{\infty} \chi(n(\lambda -c_n))  b_{+, +} (x, \lambda) e^{i\Phi(x, \lambda)} d\lambda =:  n^{1-\frac{1}{p}}e^{i\int_{0}^{x} \sqrt{-V(s)} ds} \mathrm{I}_{n, x}.
\end{align*}
Note that $|\l|\lesssim n^{-1}$ on $\supp \chi(n(\cdot -c_n))$.

To estimate $\mathrm{I}_{n,x}$, we see that
\begin{align}\label{eq:Phib_++est}
|\Phi(x,\l)|\lesssim \d^2,\quad e^{i\theta}b_{+,+}(x,\l)=2\l^{\frac{1}{2}}(-V(x))^{-\frac{1}{4}}|\mathrm{Wr}(0)|^{-1}\left(1+o(1)\right)
\end{align}
hold for $x\in \Omega_n$ and $|\l|\lesssim n^{-1}$ as $n\to \infty$, where $\theta\in \R$ is taken such that $e^{i\theta}i(\overline{\mathrm{Wr}(0)})^{-1}=|\mathrm{Wr}(0)|^{-1}$. The first inequality follows from
\begin{align*}
|\Phi(x,\l)|\lesssim \l^2\int_x^0\jap{s}^{\frac{\m}{2}}ds\lesssim \l^2\jap{x}^{1+\frac{\m}{2}}\lesssim n^{-2}\cdot \d^2\cdot n^2=\d^2
\end{align*}
for $x\in \Omega_n$ and $|\l|\lesssim n^{-1}$.
The second asymptotics in \eqref{eq:Phib_++est} follows from \eqref{eq:b++asymp2} an the observation that if $\l\lesssim n^{-1}$, then
\begin{align*}
\l^2\lesssim n^{-2}\sim \jap{x}^{-1-\m/2}=o(\jap{x}^{-\m})
\end{align*}
for $x\in \Omega_n$ as $n\to \infty$, where we recall $0<\m<2$.

Now we take $\d>0$ sufficiently small so that $\re (e^{i\theta}b_{+,+}(x,\l)e^{i\Phi(x,\l)})\gtrsim \l^{\frac{1}{2}}(-V(x))^{-\frac{1}{4}}$ for $x\in \Omega_n$, $|\l|\lesssim n^{-1}$, and sufficiently large $n$, which is possible by \eqref{eq:Phib_++est}. Then
\begin{align*}
|\mathrm{I}_{n,x}|\geq& |e^{i\theta} \mathrm{I}_{n,x}|\gtrsim \re \mathrm{I}_{n,x}\geq \frac{1}{2\pi} \int_{0}^{\infty} \chi(n(\lambda -c_n))\l^{\frac{1}{2}}(-V(x))^{-\frac{1}{4}}  d\lambda\\
\geq& \frac{(-V(x))^{-\frac{1}{4}}}{2\pi}\int_{\frac{13}{6n}}^{\frac{19}{6n}}\l^{\frac{1}{2}}d\l\sim n^{-\frac{3}{2}}\jap{x}^{\frac{\m}{4}},
\end{align*}
for $x\in \Omega_n$ and sufficiently large $n$. Therefore, $|W_{+, +, +, \g_c} u_n (x)|\gtrsim n^{1-\frac{1}{p}}n^{-\frac{3}{2}}\jap{x}^{\frac{\m}{4}}\sim n^{-\frac{2-\m}{2+\m}\cdot\frac{1}{2}-\frac{1}{p}}$ for $x\in \Omega_n$ and 
\begin{align*}
\|W_{+, +, +, \g_c} u_n\|_{L^p (\Omega_n)} \gtrsim n^{\frac{2-\m}{2+\m}\left(\frac{1}{p}-\frac{1}{2}\right)}.
\end{align*}

\underline{\textbf{Estimates on $W_{+, +, -, \g_c}$}}
We use the integral expression
\begin{align*}
W_{+, +, -, \g_c} u_n (x) =  n^{1-\frac{1}{p}} \int_{0}^{\infty} \chi(n(\lambda -c_n))  b_{+, -} (x, \lambda) e^{-iy(x, \lambda)} d\lambda,
\end{align*}
the support property of $\chi$ and \eqref{eq:lowb+-imp} to deduce the pointwise bound 
\begin{align*}
|W_{+, +, -, \g_c} u_n (x)| \lesssim n^{1-\frac{1}{p}} \int_{0}^{\frac{10}{n}} \lambda^{\frac{1}{2}} \jap{x}^{\frac{\mu}{4}} \cdot \jap{x}^{\mu -2} d\lambda \lesssim n^{-\frac{2-\m}{2+\m}\cdot\frac{1}{2}-\frac{1}{p}} \cdot n^{(\m-2)\frac{\m}{4}}
\end{align*}
for $x \in \Omega_n$, where we recall that $x<0$ for $x\in \Omega_n$. Therefore we obtain $\|W_{+, +, -, \g_c} u_n\|_{L^p (\Omega_n)} \lesssim n^{\frac{2-\m}{2+\m}\left(\frac{1}{p}-\frac{1}{2}\right)} \cdot n^{(\m-2)\frac{\m}{4}}=o(n^{\frac{2-\m}{2+\m}\left(\frac{1}{p}-\frac{1}{2}\right)})$ since we assume $\mu \in (0, 2)$. This completes the proof.
\end{proof}

\subsection{Unboundedness for $p>2$ in the low energy regime}
We fix $c>0$. Let $\g_c \in C^{\infty} _c ([0, \infty); [0, 1])$ satisfy $\g_c (\lambda)=1$ for $\lambda \in [0, c]$ and $\supp \g_c \subset [0, 2c]$.
\begin{thm}\label{26521446}
Let $p \in [2, \infty]$. The low energy wave operator $W_+ \g_c (P_0)$ is bounded on $L^p (\R)$ if and only if $p=2$.
\end{thm}

Define $M_j(x):=\int_0^x(-V(s))^{\frac{1-2j}{2}}ds$. Let $\chi \in C^{\infty} _c (\R; [0, 1])$ be such that $\supp \chi \subset [\frac{1}{2}, 2]$ and $\chi (\lambda) =1$ on $[\frac{2}{3}, \frac{5}{3}]$. By Young's inequality, we have $\|\chi (nD_x -1)\|_{L^p (\R)\to L^p(\R)} \lesssim 1$. In the following, we assume that
\begin{align*}
\c\in \left(\frac{3}{\mu +1}, \frac{2}{\mu}\right),\quad 0<\e\ll 1,
\end{align*}
which are determined later.
It suffices for proving Theorem \ref{26521446} to construct a sequence $u_n\in L^p(\R)$ such that $\|u_n\|_{L^p(\R)}\sim 1$ and
\begin{align*}
&\|W_nu_n\|_{L^p(|x+n^{\c}|\lesssim \e n^{2-\frac{\c\m}{2}})}\to \infty\,\,(n\to \infty)\quad \text{with}\quad W_n:=e^{-iM_0(x)}\cdot W_{+,+,\g_c}\chi (nD_x -1)\\
&\|W_{+,-,\g_c}\chi (nD_x -1)u_n\|_{L^p(|x+n^{\c}|\lesssim \e n^{2-\frac{\c\m}{2}})}\lesssim 1,\quad W_{-,\pm,\g_c}\chi (nD_x -1)u_n=0.
\end{align*}

Before specifying such a sequence, we study oscillating properties of $W_n$.
We note that $\g_c(\l)\chi(n\l-1)=\chi (n\l -1)$ for sufficiently large $n$ by the support properties of $\g_c$ and $\chi$. By using this and the change of variable $\l\to n^{-1}(\l+1)$, we obtain the following expression:
\begin{align}\label{eq:W_nb_n}
W_n(x, x') =n^{-1}e^{iL_{n,x}-in^{-1}x'} \int_{0}^{\infty} b_{n} (x, \lambda) e^{iL_{n,x}\Phi_{n,x,x'} (\lambda)}d\lambda,\,\, b_n(x,\l):=\chi (\l)b_{+,+}(x,(\l+1)/n),
\end{align}
where $L_{n, x} = \frac{1}{2n^2}M_1(x)$ and $\Phi_{n, x, x'} (\l):=L_{n,x}^{-1}(y(x,(\l+1)/n)-(\l+1) x'/n-M_0(x))$. In view of Taylor's expansion 
\begin{align}\label{eq:yTay}
y(x,\l)=M_0(x)+\frac{1}{2}M_1(x)\l^2\underbrace{-\l^4\int_{0}^{x} \frac{ds}{2\sqrt{-V(s)} (\sqrt{-V(s)} + \sqrt{\lambda^2 -V(s)})^2}}_{=:\mathcal{R}_{x}(\l)},
\end{align}
we rewrite $\Phi_{n,x,x'}(\l)$ as
\begin{align}\label{eq:Phinx}
\Phi_{n,x,x'}(\l)=\l^2-2(nx' M_1(x)^{-1}-1)\l+\mathcal{R}_{n,x}(\l),\quad \mathcal{R}_{n,x}(\l):=L_{n,x}^{-1}\mathcal{R}_{x}((\l+1)/n).
\end{align}

The next lemma shows that the term $\mathcal{R}_{n,x}(\l/n)$ can be regarded as a perturbation of $\frac{1}{2}M_1(x)(\l/n)^2$. We also derive upper bounds of $\pa_xL_{n,x}$ and $\pa_{\l}\pa_x\Phi_{n,x,x'}(\l)$ which are used later.

\begin{lemma}\label{lem:lowRperturbMd}
\noindent$(i)$
For $\a\in\mathbb{N}_0$, we have $|\pa_{\l}^{\a}\mathcal{R}_{n,x}(\l)|\ll 1$ when $\l\sim 1$ and $|x|\ll n^{\frac{2}{\m}}$ with sufficiently large $n$.

\noindent$(ii)$ We have $L_{n,x}\sim -n^{(1+\frac{\m}{2})\c-2}$,  $|\pa_xL_{n,x}|\lesssim n^{\frac{\c\m}{2}-2}$, $|\pa_x\Phi_{n,x,x'}(\l)|\lesssim n^{-\c}$, and $|\pa_{\l}\pa_x\Phi_{n,x,x'}(\l)|\lesssim n^{-\c}$ when $x\sim -n^{\c}$, $x'\sim -n^{\c(1+\frac{\m}{2})-1}$, and $\l\sim 1$ hold with sufficiently large $n\in\mathbb{N}$.

\end{lemma}

\begin{proof}
\noindent$(i)$
By the definition of $\mathcal{R}_{x}(\l)$, we have $|\pa_{\l}^{\a}\mathcal{R}_{x}(\l)|\lesssim \l^{4-\a}|M_2(x)|$ uniformly in $\l>0$ and $x\in\mathbb{R}$, where we recall $M_j(x)=\int_0^x(-V(s))^{\frac{1-2j}{2}}ds$. By \eqref{26641742} and \eqref{eq:Vneg}, we have $|M_2(x)|\sim \jap{x}^{\m}|M_1(x)|$ and $|\pa_{\l}^{\a}\mathcal{R}_{x}(\l)|\ll \l^{2-\a}|M_1(x)|$ when $\l\ll \jap{x}^{-\m/2}$. By using this and the definition of $\mathcal{R}_{n,x}=L_{n,x}^{-1}\mathcal{R}_{x}((\l+1)/n)$ with $L_{n,x}=\frac{1}{2n^2}M_1(x)$, the claim follows.

\noindent$(ii)$ We recall $L_{n,x}=\frac{1}{2n^2}M_1(x)$. Then we see that $L_{n,x}\sim -n^{(1+\frac{\m}{2})\c-2}$ and $|\pa_xL_{n,x}|=\frac{1}{2n^2}|V(x)|^{-\frac{1}{2}}\lesssim n^{\frac{\c\m}{2}-2}$, where we note that $M_1(x)<0$ for $x<0$. By \eqref{eq:yTay},
\begin{align*}
|\pa_x(y(x,\l)-M_0(x))|=\left|\frac{1}{2}M_1'(x)\l^2-\l^4\frac{1}{2\sqrt{-V(x)} (\sqrt{-V(x)} + \sqrt{\lambda^2 -V(x)})^2}\right|\lesssim \l^2\jap{x}^{\frac{\m}{2}}+\l^4\jap{x}^{\frac{3}{2}\m}.
\end{align*}
Hence
\begin{align*}
|\pa_x\Phi_{n, x, x'} (\l)|\lesssim& |L_{n,x}^{-1}||(\pa_xy)(x,(\l+1)/n)-M_0'(x)|+|L_{n,x}^{-2}\pa_xL_{n,x}||y(x,(\l+1)/n)-(\l+1)x'/n-M_0(x)|\\
\lesssim& n^{-\c}.
\end{align*}

By differentiating $\Phi_{n, x, x'} (\l)=L_{n,x}^{-1}(y(x,(\l+1)/n)-(\l+1) x'/n-M_0(x))$ with respect to $\l$ and $x$, we have $\pa_{\l}\pa_x\Phi_{n, x, x'} (\l)=L_{n,x}^{-1}\pa_{x}\pa_{\l}(y(x,(\l+1)/n))+(\pa_xL_{n,x}^{-1})\pa_{\l}(y(x,(\l+1)/n))-(\pa_xL_{n,x}^{-1})x'/n $. Since
\begin{align*}
|\pa_{\l}y(x,\l)|=\left|\int_0^x\frac{\l}{\sqrt{\l^2-V(s)}}ds\right|\lesssim \l\jap{x}^{1+\frac{\m}{2}},\,\, |\pa_x\pa_{\l}y(x,\l)|=\frac{\l}{\sqrt{\l^2-V(x)}}\lesssim \l\jap{x}^{\frac{\m}{2}}
\end{align*}
we obtain
\begin{align*}
|\pa_{\l}\pa_x\Phi_{n, x, x'} (\l)|\lesssim |L_{n,x}^{-1}|n^{-2}\jap{x}^{\frac{\m}{2}}+|L_{n,x}^{-2}\pa_xL_{n,x}|(n^{-2}\jap{x}^{1+\frac{\m}{2}}+n^{-1}\jap{x'})\lesssim n^{-\c},
\end{align*}
which completes the proof.
\end{proof}

For $0<C_1 < C_2$, we set 
\begin{align*}
\tilde{\Omega}_{n} := \{(x, x') \in \R^2 \mid C_1 < nx'M_1(x)^{-1} -1 < C_2\}
\end{align*}
for sufficiently large $n \in \N$. 
We note that
\begin{align}\label{eq:lowxvolasy}
x\sim -n^{\c},\,(x, x') \in \tilde{\Omega}_{n} \Rightarrow x' \sim -n^{\c (1+\frac{\mu}{2}) -1}\,\text{and }\,\VOL (\{x' \in \R \mid (n^{\c}, x') \in \tilde{\Omega}_{n} \}) \sim n^{\c (1+\frac{\mu}{2}) -1}
\end{align}
for sufficiently large $n\in\mathbb{N}$.

\begin{proposition}\label{prop:p>2lowstphase}
For $R_2\gg R_1\gg 1$, $0<C_1\ll 1\ll C_2$ and sufficiently large $n\in\mathbb{N}$, we have the following:

\vspace{1mm}
\noindent$(i)$  We have
\begin{align*}
|W_{n}(x,x')|\lesssim n^{-\frac{3}{2}}L_{n,x}^{-1}\jap{x}^{\frac{\m}{4}}
\end{align*}
if $x\sim -n^{\c}$, $(x,x'),(n^{\c},x')\in \tilde{\Omega}_n^c$ and $n\in\mathbb{N}$ is sufficiently large.

\vspace{1mm}
\noindent$(ii)$ There exists a unique non-degenerate critical point $\l_{n,x,x'}\in (C_1/2,2C_2)$ of $\Phi_{x, x'} (\lambda)$  such that
\begin{align*}
W _n (x, x')=n^{-1}e^{iL_{n,x}-in^{-1}x'}\cdot |L_{n,x}|^{-\frac{1}{2}} b_n (x, \l_{n,x, x'}) e^{iL_{n,x}\Phi_{n,x, x'} (\l_{n,x, x'})-\frac{i\pi}{4}}+ O(n^{-\frac{3}{2}}L_{n,x}^{-1}\jap{x}^{\frac{\m}{4}})
\end{align*}
for $x\sim -n^{\c}$ and $(x,x')\in \tilde{\Omega}_n$ with sufficiently large $n\in\mathbb{N}$.

\end{proposition}

\begin{proof}

\noindent$(i)$ We recall that $\supp b_n(x,\cdot)\subset [\frac{1}{2},2]$ from \eqref{eq:W_nb_n}. Therefore, if we take $C_1<\frac{1}{2}<2<C_2$, then we have $|\pa_x\Phi_{n,x,x'}(\l)|\gtrsim 1$ when $x\sim n^{\c}$, $(x,x'),(n^{\c},x')\in \tilde{\Omega}_n^c$ and $n\in\mathbb{N}$ is sufficiently large. Thus, the claim follows from performing the integration by parts in \eqref{eq:W_nb_n} and the symbolic estimates of $b_{+,+}$ given in Proposition \ref{prop:waveexpression2}.

\noindent$(ii)$ By the assumption $\c<2/\m$, \eqref{eq:Phinx}, and Lemma \ref{lem:lowRperturbMd}, we find that if $0<C_1 \ll 1\ll C_2$, and $n\in\mathbb{N}$ is sufficiently large, 
\begin{align}\label{eq:Philamder}
\pa^2 _{\l} \Phi_{n,x, x'} (\lambda)\gtrsim 1, \quad \pa_{\l} \Phi_{n,x, x'} (C_1) <0, \quad \pa_{\l} \Phi_{n,x, x'} (\lambda) (C_2) >0, \quad |\pa^{\alpha} _{\l}\Phi_{n,x, x'} (\lambda)| \lesssim 1
\end{align}
for $x\sim -n^{\c}$, $(x,x')\in \tilde{\Omega}_n$, and $\l\sim 1$ with sufficiently large $n\in\mathbb{N}$. Therefore the intermediate value theorem yields the unique existence of $\lambda_{n,x, x'} \in (C_1/2, 2C_2)$ such that $\pa_{\l} \Phi_{n,x, x'} (\l_{n,x, x'}) =0$. On the other hand, we have $L_{n,x}=\frac{1}{2n^2}M_1(x)\sim -|x|^{1+\frac{\m}{2}}n^{-2}\sim -n^{\c(1+\frac{\m}{2})-2}\to -\infty$ as $n\to \infty$ due to the assumption $\c> 3/(\mu +1)$. Thus, the claim follows from the stationary phase theorem \cite[Theorem 7.7.5]{Hor} with the large parameter $L_{n,x}$, where we note that $\l\sim 1$ holds for $\l\in \supp b_n(x,\cdot)$ due to the cut-off function $\chi(\l)$ in \eqref{eq:W_nb_n}. 
\end{proof}

\begin{proof}[Proof of Theorem \ref{26521446}]
We define
$u_n (x'):= V^{-\frac{1}{p}} _n 1_{\tilde{\Omega}_{n}} (n^{\c}, x') e^{-\frac{i\pi}{4}-iL_{n, n^{\c}} \Phi_{n, n^{\c}, x'} (\lambda_{n, n^{\c}, x'}) + in^{-1} x'}$, where $V_n = n^{\c (1+\frac{\mu}{2}) -1}$. Then we see $\|u_n\|_{L^1}\lesssim V_n^{1-\frac{1}{p}}$ and $W_{-,\pm,\g_c}\chi (nD_x -1)u_n=0$. Thus, it suffices to prove $\|W_nu_n\|_{L^p(|x+n^{\c}|\lesssim \e n^{2-\frac{\c\m}{2}})}\to \infty$ and $\|W_{+,-,\g_c}\chi (nD_x -1)u_n\|_{L^p(|x+n^{\c}|\lesssim \e n^{2-\frac{\c\m}{2}})}\lesssim 1$ as $n\to \infty$.

\underline{\textbf{Estimates on $W_n $}}
By Proposition \ref{prop:p>2lowstphase} $(i)$ and $(ii)$,
\begin{align*}
\left|W_nu_n (x)\right| =&n^{-1}|L_{n,x}|^{-\frac{1}{2}}V^{-\frac{1}{p}} _n\left|\int_{\R}b_n(x,\l_{n,x,x'})e^{i\Psi_{n,x,x'} }1_{\tilde{\Omega}_{n}} (x, x')1_{\tilde{\Omega}_{n}} (n^{\c}, x') dx' \right|\\
&+O(n^{-\frac{3}{2}}|L_{n,x}|^{-1}\jap{x}^{\frac{3}{2}}\|u_n\|_{L^1})
\end{align*}
for $x\sim -n^{\c}$ with $n\gg 1$, where we set $\Psi_{n,x,x'}=L_{n,x}\Phi_{n,x, x'} (\l_{n,x, x'})-L_{n, -n^{\c}} \Phi_{n, n^{\c}, x'} (\lambda_{n, -n^{\c}, x'})$.

We show that
\begin{align}\label{eq:lowphasecancel}
|\Psi_{n,x,x'}|\lesssim\e
\end{align}
for $|x+n^{\c}|\leq \e n^{2-\frac{\c\mu}{2}}$, $(x, x') \in \tilde{\Omega}_{n}$, and $n\gg 1$. We note that $n^{2-\frac{\c\mu}{2}}\ll n^{\c}$ by the assumption $\c>3/(\m+1)$ and hence $x\sim -n^{\c}$ in this region.
 Since $\l_{n,x,x'}$ is the unique solution to $(\pa_{\l}\Phi_{n,x,x'})(\l_{n,x,x'})=0$, we have 
\begin{align*}
|\pa_x\l_{n,x,x'}|=|(\pa_{\l}\pa_x\Phi_{n,x,x'})(\l_{n,x,x'})||(\pa_{\l}^2\Phi_{n,x,x'})(\l_{n,x,x'})|^{-1}\lesssim n^{-\c}
\end{align*}
by Lemma \ref{lem:lowRperturbMd} $(ii)$.
By the definition of $\Phi_{n,x,x'}$ (just after \eqref{eq:W_nb_n}) and Lemma \ref{lem:lowRperturbMd} $(ii)$, we see
\begin{align*}
&|\Phi_{n,x,x'}(\l_{n,x,x'})|+|\Phi_{n,x,x'}(\l_{n,-n^{\c},x'})|\lesssim 1,\,\, |L_{n,x}-L_{n,-n^{\c}}|\lesssim n^{-2}\left|\int_{-n^{\c}}^x\jap{s}^{\frac{\m}{2}}ds \right|\lesssim n^{\frac{\c\m}{2}-2}|x+n^{\c}|\leq \e.
\end{align*}
Now, the left hand side in \eqref{eq:lowphasecancel} is bounded by 
\begin{align*}
&|(L_{n,x}-L_{n,-n^{\c}})\Phi_{n,x, x'} (\l_{n,x, x'})|+|L_{n,-n^{\c}}||\Phi_{n,x, x'} (\l_{n,x, x'})- \Phi_{n, -n^{\c}, x'} (\lambda_{n,x, x'})|\\
&+|L_{n,-n^{\c}}||\Phi_{n,-n^{\c}, x'} (\l_{n,x, x'})- \Phi_{n, -n^{\c}, x'} (\lambda_{n,-n^{\c}, x'})|.
\end{align*}
Since $x\sim -n^{\c}$, $x'\sim -n^{\c(1+\frac{\m}{2})-1}$, $|L_{n,n^{\c}}|\lesssim n^{\c\left(1+\frac{\m}{2}\right)-2}$ in our region, the intermediate value theorem with the bounds proved in Lemma \ref{lem:lowRperturbMd} $(ii)$ and \eqref{eq:Philamder}, implies that this quantity is bounded by
\begin{align*}
n^{-2}n^{\frac{\c\m}{2}}|x-n^{\c}|\cdot 1+n^{\c\left(1+\frac{\m}{2}\right)-2}\cdot n^{-\c}|x-n^{\c}|+n^{\c\left(1+\frac{\m}{2}\right)-2}\cdot 1\cdot n^{-\c}|x-n^{\c}|\lesssim n^{\frac{\c\m}{2}-2}|x-n^{\c}|\leq \e,
\end{align*}
which proves \eqref{eq:lowphasecancel}.

 Proposition \ref{prop:p>2lowstphase} $(ii)$ implies $\l_{n,x,x'}\sim 1$ and hence $|(\l_{n,x,x'}+1)/n|\lesssim n^{-1}$. Since we have chosen $\c<2/\mu$, it holds that $|(\l_{n,x,x'}+1)/n|=o(\jap{x}^{-\mu/2})$ for $|x+n^{\c}|\leq n^{2-\frac{\c\mu}{2}}$ as $n\to \infty$. Taking $\theta\in \R$ such that $e^{i\theta}i(\overline{\mathrm{Wr}(0)})^{-1}=|\mathrm{Wr}(0)|^{-1}$ and applying \eqref{eq:b++asymp2} in the definition of $b_n$ given in \eqref{eq:W_nb_n}, we obtain
\begin{align*}
e^{i\theta}b_n(x,\l_{n,x,x'})= 2|\mathrm{Wr}(0)|^{-1}\chi(\l_{n,x,x'}) n^{-\frac{1}{2}}\l_{n,x,x'}^{\frac{1}{2}}(-V(x))^{-\frac{1}{4}}(1+o(1)).
\end{align*}
Taking $\e>0$ sufficiently small and $n$ sufficiently large, we have
\begin{align*}
\re (e^{i\theta}b_n(x,\l_{n,x,x'})e^{i\Psi_{n,x,x'}})\gtrsim \chi(\l_{n,x,x'}) n^{-\frac{1}{2}}\jap{x}^{\frac{\mu}{4}}
\end{align*}
by $\l_{n,x,x'}\sim 1$, $|V(x)|\sim \jap{x}^{-\mu}$, and \eqref{eq:lowphasecancel}.
Therefore,
\begin{align*}
\left|W_nu_n (x)\right|\gtrsim&\re\left(e^{i\theta} W_nu_n (x)\right)\gtrsim n^{-1}|L_{n,x}|^{-\frac{1}{2}}V^{-\frac{1}{p}}_n\cdot n^{-\frac{1}{2}}\jap{x}^{\frac{\m}{4}} \left|\int_{\R}\chi(\l_{n,x,x'})1_{\tilde{\Omega}_{n}} (x, x')1_{\tilde{\Omega}_{n}} (n^{\c}, x')dx' \right|\\
&+O(n^{-\frac{3}{2}}|L_{n,x}|^{-1}\jap{x}^{\frac{\m}{4}}\|u_n\|_{L^1})\\
\gtrsim& n^{-\frac{3}{2}}|L_{n,x}|^{-\frac{1}{2}}V^{1-\frac{1}{p}}_n\jap{x}^{\frac{\m}{4}} \sim n^{\frac{\c \mu}{4} -1} \cdot n^{\{\c (1+\frac{\mu}{2}) -1\}(\frac{1}{2} -\frac{1}{p})}
\end{align*}
for $x\sim -n^{\c}$ with sufficiently large $n$, where we have used $\|u_n\|_{L^1}\lesssim V_n^{1-\frac{1}{p}}n^{-1}\jap{x}^{1+\frac{\m}{2}}$ and $\left|\int_{\R}\chi(\l_{n,x,x'})1_{\tilde{\Omega}_{n}} (x, x')1_{\tilde{\Omega}_{n}} (n^{\c}, x')dx' \right|\sim V_n$ 
holds if $|x+n^{\c}| \le \epsilon n^{2-\frac{\c \mu}{2}}$. Taking $L^p$ norms on the both sides, we have
\begin{align*}
\|W_n u_n\|_{L^p (\{|x+n^{\c}| \le \epsilon n^{2-\frac{\c \mu}{2}}\})} &\gtrsim n^{\frac{\c \mu}{4} -1} \cdot n^{\{\c (1+\frac{\mu}{2}) -1\}(\frac{1}{2} -\frac{1}{p})} \cdot n^{(2-\frac{\c \mu}{2})\cdot \frac{1}{p}} \\
&\sim n^{(\frac{1}{2} -\frac{1}{p})(\c \mu + \c -3)}\to \infty\quad \text{as $n\to \infty$}
\end{align*}
by the assumption $\c > \frac{3}{\mu +1}$.

\underline{\textbf{Estimates on $W _{+, +, -, \g_c}$}}
We rewrite $W_{+, +, -, \g_c} \chi (nD_x -1) u_n (x)$ as
\begin{align*}
|W_{+, +, -, \g_c} \chi (nD_x -1) u_n (x)| = \left| \int_{\R} \int_{0}^{\infty} \chi (n\l -1) b_{+, -} (x, \l) e^{-iy(x, \l)-i\l x'} d\l u_n(x') dx' \right|.
\end{align*}
Now, by \eqref{eq:lowb+-imp}, the integral in $\lambda$ can be estimated from above as
\begin{align*}
\left| \int_{0}^{\infty} \chi (n\l -1) b_{+,-}(x, \l) e^{-iy(x, \l)-i\l x'} d\l \right| \lesssim \int_{0}^{\frac{10}{n}} \l^{\frac{1}{2}} \jap{x}^{\frac{5\mu}{4} -2} d\l \lesssim \jap{x}^{\frac{5\mu}{4} -2} n^{-\frac{3}{2}}
\end{align*}
for $x<0$.
Therefore, by the support property of $u_n$, we obtain
\begin{align*}
&|W_{+, +, -, \g_c} \chi (nD_x -1) u_n (x)| \lesssim \jap{x}^{\frac{5\mu}{4} -2} n^{-\frac{3}{2}} \|u_n \|_{L^1} \lesssim V^{1-\frac{1}{p}} _n \jap{x}^{\frac{5\mu}{4} -2} n^{-\frac{3}{2}}, \\
&\|W_{+, +, -, \g_c} \chi (nD_x -1) u_n\|_{L^p (\{|x+n^{\alpha}| \le \epsilon n^{2-\frac{\alpha \mu}{2}}\})} \lesssim V^{1-\frac{1}{p}} _n \cdot n^{(\frac{5\mu}{4} -2)\alpha -\frac{3}{2}} \cdot n^{(2-\frac{\alpha \mu}{2})\cdot \frac{1}{p}} = O(1)
\end{align*}
if we take $0 < \alpha - \frac{3}{\mu +1} \ll 1$.
\end{proof}

\end{document}